\documentclass[10pt]{amsart}

\usepackage{ytableau}
\usepackage{tikz-cd}
\usepackage{tikz}
\usepackage{amssymb}
\usepackage{mathtools}
\usepackage{extarrows}                
\usepackage{geometry}

\usepackage[greek, english]{babel}

\newcommand{\al}{\alpha}

\newcommand{\la}{\lambda}
\newcommand{\La}{\Lambda}

\DeclareMathOperator{\st}{SST_{\la}(\mu)}

\DeclareMathOperator{\im}{Im}
\DeclareMathOperator{\Sh}{Sh}
\DeclareMathOperator{\Hlm}{H_1(\la,\mu)}
\DeclareMathOperator{\Hlmm}{H_2(\la,\mu)}

\DeclareMathOperator{\Hom}{Hom}

\DeclareMathOperator{\D}{\mathrm{\Delta}}
\DeclareMathOperator{\Sp}{\mathrm{Sp}}

\DeclareMathOperator{\chK}{\mathrm{chK}}
\DeclareMathOperator{\spn}{\mathrm{span}}

\newtheorem{theorem}{Theorem}[section]
\newtheorem{lemma}[theorem]{Lemma}
\newtheorem{proposition}[theorem]{Proposition}
\newtheorem{corollary}[theorem]{Corollary}

\newtheorem{definition}[theorem]{Definition}

\theoremstyle{remark}
\newtheorem{remark}[theorem]{Remark}

\newtheorem{example}[theorem]{Example}

\numberwithin{equation}{section}
\newcommand{\rectangle}[2]{{
\fboxsep=-\fboxrule\sbox0{}\wd0=#1\ht0=#2\relax\fbox{\box0}}}
\newcommand{\bx}{\rectangle{4pt}{8pt}} 

\DeclareMathOperator{\dts}{,... ,}

\newcommand{\sm}{\sum\limits}

\newcommand{\sts}{\cdot \cdot \cdot}

\newcommand{\ots}{\otimes ... \otimes}

\newcommand{\Lnr}{\Lambda^{+}(n,r)}

\newcommand{\Blm}{B(\la,\mu)}

\begin{document}


		
		\begin{abstract}Let $G=GL_n(K)$ be the general linear group defined over an infinite field $K$ of positive characteristic $p$ and let  $\Delta(\lambda)$ be the Weyl module of $G$ which corresponds to a partition $\lambda$. In this paper we classify all homomorphisms $\Delta(\lambda) \to \Delta(\mu)$ when $\lambda=(a,b,1^d)$ and $\mu=(a+d,b)$, $d>1$. In particular, we show that   $\Hom_G(\Delta(\lambda),\Delta(\mu))$ is nonzero if and only if $p=2$ and $a$ is even. In this case, we show that the dimension of the homomorphism space is equal to 1 and we provide an explicit generator whose description depends on binary expansions of various integers. We also show that these generators in general are not compositions of Carter-Payne homomorphisms.
		\end{abstract}


\title[On homomorphisms between Weyl modules a column transposition]{On homomorphisms between Weyl modules: The case of a column transposition}
	
\author[Evangelou]{Charalambos Evangelou}
\address{Department of Mathematics, University of Athens, Greece}
\email{cevangel@math.uoa.gr}



\date{\today}
	

\maketitle

\section{Introduction}  Let K be an infinite field of positive characteristic p and let $G = GL_n(K)$ be the general linear group defined over $K$. We let $\Delta(\lambda)$ denote the Weyl module of  corresponding to a partition $\lambda$. These modules are of central importance
in the polynomial representation theory of $G$ (\cite{green2006polynomial}, \cite{jantzen2003representations}). Since the classical papers of
Carter-Lusztig \cite{CL} and Carter-Payne \cite{CP}, homomorphism spaces $\Hom_G(\Delta(\lambda),\Delta(\mu))$ between Weyl modules have attracted much attention. In those works sufficient arithmetic conditions on $\lambda,\mu$ and p were found so that $
\Hom_G(\Delta(\la),\Delta(\mu))\neq 0$. In general the determination of the dimensions of the homomorphism spaces, or even when they are nonzero, seems a difficult problem.
There seem to be very few general results. Cox and Parker \cite{CoxParkerGL3} classified homomorphisms between Weyl modules for $GL_3(K)$. Ellers and Murray \cite{ellers2010carter} and also
Kulkarni \cite{kulkarni2006ext} showed that the dimension of $\Hom_G(\Delta(\lambda),\Delta(\mu))$ is equal to 1 for $\lambda,\mu$ a pair of Carter-Payne partitions for which the number of raised boxes is 1. Maliakas and Stergiopoulou obtained a classification result for $\Hom_G(\Delta(\lambda),\Delta(\mu))$ when one of $\lambda,\mu$ is a hook. By the work of Dodge \cite{dodge2011large} and Lyle \cite{lyle2013large} it is known that these homomorphism spaces may have arbitrarily large dimensions. A non vanishing result for homomorphisms of Weyl modules was obtained in \cite{evangelou2023stability}. In this paper we classify all homomorphisms $\Delta(\lambda) \longrightarrow \Delta(\mu)$, when $\lambda = (a,b,1^d) $ and $ \mu =(a+d,b)$ (Theorem \ref{Mainresult}). In terms of diagrams, the bottom d boxes of the first column of the diagram of $\lambda$ are raised to the first row. In particular, we prove
that $\Hom_G(\Delta(\lambda),\Delta(\mu))$ is nonzero if and only if $p = 2$ and a is even. In this case, we show that the dimension of $\Hom_G(\Delta(\lambda),\Delta(\mu))$  is equal to 1 and we provide an explicit generator whose description depends on binary expansions of various
integers associated to a,b,d. We also show that these generators in general are not
compositions of Carter-Payne homomorphisms. Our techniques rely on a systematic use of the classical presentation of the Weyl modules $\Delta(\lambda)$ and utilize combinatorics and computations involving tableaux and properties of binomial coefficients mod 2. 

The paper is organized as follows. In Section 2 we establish notation and gather
preliminaries that will be needed later. In Section 3 we state the main result and
introduce notions associated to binary expansions that are required to describe a
generator of the homomorphism space. In Section 4 the strategy of the proof is
outlined and some reductions are made in the resulting linear systems. In sections
5-8 we consider the four cases depending on the parity of a and b. The substantial
case is when a is even. In Section 9 we show by means of example that our maps are
not in general compositions of Carter-Payne homomorphisms, Example \ref{NonCarter-Payne}. Also
we find equivalent conditions so that our generator of $\Delta(\lambda) \longrightarrow \Delta(\mu)$ corresponds to the sum of all semistandard tableaux of shape $\mu$ and weight $\la$, Corollary \ref{M-D2023-C}.

\section{Recollections}
The purpose of this section is to establish notation and gather facts that will be needed in the sequel.
\subsection{Notation}
Let K be an infinite field. We will be working with homogeneous polynomial representations of $G=GL_n(K)$, or equivalently, with modules over the Schur algebra $S = S_K(n,r)$. A standard reference here is \cite{green2006polynomial}. For a positive integer $r$, let $\Lambda(n,r)$ be the set of sequences $\al=(\al_1, \dots, \al_n)$ of length $n$ of nonnegative integers that sum to $r$ and  $\Lambda^+(n,r)$ the subset of $\Lambda(n,r)$ consisting of partitions, that is sequences $\mu=(\mu_1, \dots, \mu_n)$ such that $\mu_1 \ge \mu_2 \ge \dots \ge \mu_n$. The length  of a partition $\mu$ is the number of nonzero parts of $\mu$ and is denoted by $\ell(\mu)$.  

If $\al, \beta \in \Lambda(n,r)$, where $\al=(\al_1,\dots, \al_n), \beta = (\beta_1,\dots,\beta_n)$, we write $\al \unlhd \beta$ if $\al_1+\dots+\al_i \le \beta_1+\dots+\beta_i$ for all $i$.

We denote by $V=K^n$ be the natural $G$-module, $DV=\sum_{i\geq 0}D_iV$ the divided power algebra of $V$ and $\Lambda V =\sum_{i\geq 0}\Lambda_iV$  the exterior hopf algebra. The divided power algebra is defined as the graded dual of the Hopf algebra $S(V^*)$, where $V^*$ is the linear dual of $V$ and $S(V^*)$ is the symmetric algebra of $V^*$ (see [2], I.4). If $\la=(\la_1,\dots, \la_n) \in \Lambda(n,r)$, let $D(\la)$ (respectively $\Lambda(\la)$) be the tensor product $D_{\la_1}V\otimes \dots \otimes D_{\la_n}V$. If $\la \in \Lambda^+(n,r)$ and $\la'$ its conjugate partition, we denote (defined in \cite[Def.II.1.3]{akin1982schur}) by $\Delta(\la)$ the corresponding Weyl module as the image of the map 
\begin{center}
$d_{\la}' : D(\la) \longrightarrow \Lambda(\la')$.   
\end{center}

\subsection{Semistandard tableaux}
Consider the order $e_1 < e_2 <\dots< e_n$ on the natural basis $\{e_1,\dots, e_n\}$ of $V$. In the sequel we will denote each element $e_i$ by its subscript $i$. For a partition $\mu=(\mu_1,\dots,\mu_n) \in \Lambda^+(n,r)$, a tableau of shape $\mu$ is a filling of the diagram of $\mu$ with entries from $\{1,\dots,n\}$. A tableau is called \textit{row semistandard} if the entries are weakly increasing across the rows from left to right.  A row semistandard tableau is called \textit{semistandard} if the entries are strictly increasing in the columns from top to bottom. The set consisting of the semistandard (respectively, row semistandard) tableaux of shape $\mu$ will be denoted by $\mathrm{SST}(\mu)$ (respectively, $\mathrm{RSST}(\mu)$). The \textit{weight} of a tableau $T$ is the tuple $\alpha=(\alpha_1,\dots,\alpha_n)$, where $\alpha_i$ is the number of appearances of the entry $i$ in $T$. The set consisting of the semistandard (respectively, row semistandard) tableaux of shape $\mu$ and weight $\alpha$ will be denoted by $\mathrm{SST}_{\alpha}(\mu)$ (respectively, $\mathrm{RSST}_{\alpha}(\mu)$). For example, the following tableau of shape $\mu=(6,4)$
\[T= \begin{ytableau}
		\ 1&1&2&2&3&4\\
		\ 2&2&4&4 
	\end{ytableau}, \]
is semistandard and has weight $\alpha=(2,4,1,3)$. We will use 'exponential' notation for row semistandard tableaux. For the previous example we may write \[ T=\begin{matrix*}[l]
	1^{(2)} 2^{(2)} 3 4 \\
	2^{(2)}4^{(2)}  \end{matrix*}. \] 
To each row semistandard tableau T of shape $\mu = (\mu_1 \dts \mu_n)$ we may associate the element
\begin{center}
$x_T = x_T(1) \ots x_T(n) \in D(\mu),$     
\end{center}
where $x_T(i) = 1^{(a_{i1})} \sts n^{(a_{in})}$ and $a_{ij}$ is equal to the number of appearances of $j$ in the $i$-th row of $T$. For example the $T$ depicted above yields $x_T = 1^{(2)}2^{(2)}3^{(1)}4^{(1)} \otimes 2^{(2)}4^{(2)}$. \ 

 Henceforth, we define $[T] = d_{\mu}'(x_T) $ for each row semistandard tableau $T$.

\subsection{Relations for Weyl modules} In this subsection we summarise some fundamental results on Weyl modules. We start with two theorems which are contained in \cite{akin1982schur}.\ 

\begin{definition}
Let $\la=(\la_1,\dots,\la_n) \in \Lnr$ and $s, t$ are integers such that $1 \le s \le n-1$ and $1 \le t \le \la_{s+1} $. We denote the sequence $(\la_1,\dots,\la_s+t, \la_{s+1}-t,\dots,\la_n) \in \La(n,r)$ by $\la({s,t})$. The following S-map is defined
\begin{equation}\label{pres1}
\bx_{\la}: \bigoplus_{s=1}^{n-1}\bigoplus_{t=1}^{\lambda_{s+1}}D(\la(s,t)) \longrightarrow D(\lambda) 
\end{equation}
where the restriction of $\bx_{\la}$ to the summand $D(\la(s,t))$ is the composition
\begin{align}
\rectangle{5pt}{9pt}_{\la, s,t}:D(\la(s,t)) &\xrightarrow {1\otimes \cdots \otimes \D \otimes \cdots \otimes 1} D(\lambda_1,\dots,\lambda_s,t,\lambda_{s+1}-t,\dots,\lambda_n) \nonumber \\ 
&\xrightarrow{1\otimes\cdots \otimes \eta \otimes \cdots \otimes 1} D(\lambda), \label{boxmap}    
\end{align}
where $\D:D(\lambda_s+t) \to D(\lambda_s,t)$ and $\eta:D(t,\lambda_{s+1}-t) \to D(\lambda_{s+1})$ are the indicated components of the comultiplication and multiplication respectively of the Hopf algebra $DV$.

\end{definition}

\begin{theorem}\label{BAWbase}\cite[Section II.3]{akin1982schur}
The set $\{[T] : T \in SST(\mu)\}$ is a basis of the K-vector space $\D(\mu)$.   
\end{theorem}

\begin{theorem}\label{BAWpresentation}\cite[Section II.3]{akin1982schur}

\begin{enumerate}
    
    \item[$1)$]
There is isomorphism  $\theta_{\la}:  D(\la)/\im\bx_{\la} \longrightarrow \D(\la), \ x + Im\bx_{\la} \longmapsto d_{\la}'(x)$.

    \item[$2)$]
There is an exact sequence of $S$-modules    

\begin{equation}\label{pres1}
\bigoplus_{s=1}^{n-1}\bigoplus_{t=1}^{\lambda_{s+1}}D(\la(s,t)) \xlongrightarrow{\bx_{\la}} D(\lambda) \xrightarrow{\pi_{\D(\la)}} \D(\lambda) \to 0,
\end{equation}
where $\pi_{\D(\la)} = \theta_{\la} \circ \pi$ and $\pi: D(\la) \longrightarrow D(\la)/\im\bx_{\la}$ is the natural projection.

\end{enumerate}
\end{theorem}

\begin{definition}
Let $\la \in \La(n,r), \ \mu \in \La^{+}(2,r)$ and $T\in RSS_{\la}(\mu)$. We note that $\la_i = a_i + b_i$ for each i, where  $a_i$ (respectively, $b_i$) is the number of appearances of i in the first
row (respectively, second row) of T. Now we  define the $S$-map $\phi_T : D(\la) \longrightarrow \D(\mu)$ as the following composition
\begin{align*}
D(\la) &\xlongrightarrow{\D_1 \ots \D_n} (D_{a_1}V \otimes D_{b_1}V) \ots  (D_{a_n}V \otimes D_{b_n}V)\\
&\xlongrightarrow{\simeq} ( D_{a_1}V \ots D_{a_n}V) \otimes  (D_{b_1}V \ots D_{b_n}V)
\xlongrightarrow{\eta_1 \otimes \eta_2 } D(\mu) \xlongrightarrow{d_{\mu}'} \D(\mu),
\end{align*}

\noindent where $\eta_{1} :  D_{a_1}V \ots D_{a_n}V \longrightarrow D_{\mu_1}V$ (respectively $\eta_2$), $\D_i : D_{\la_i}V \longrightarrow D_{a_i}V \otimes D_{b_i}V$ is the indicated components of the multiplication, comultiplication, respectively, of Hopf algebra $DV$. 
\end{definition}

\begin{remark}\label{remarkphi}
	We note that $\phi_T(1^{(\la_1)} \otimes \cdots \otimes n^{(\la_n)} )=[T]$ if $T \in \mathrm{RSST}_{\la}(\mu)$, where $\la = (\la_1,\dots,\la_n) \in \Lambda(n,r)$. 
\end{remark}

 According to \cite[Section 2, equation (11)]{akin1985characteristic}, we have the following proposition

\begin{proposition}\label{Base}\cite[Section 2]{akin1985characteristic}
Let $\la \in \Lambda(n,r)$ and $\mu \in \Lambda^+(2,r)$, then  the set $\{\phi_T :T \in \st\}$, constitutes  a basis of $\Hom_S(D(\la), \D(\mu))$.
\end{proposition}

 A very useful proposition that will make the computations easier later on, is the following

\begin{proposition}\cite[Section 2]{akin1985characteristic}\label{Cycle}
$D(\la)$ is cyclic S-module, generated by the element $1^{(\la_1)} \otimes \cdots \otimes    n^{(\la_n)}$.
\end{proposition}

\subsection{Binomial coefficients}
In our work, we will heavily rely on Lucas's Theorem from classical number theory.

\begin{theorem}\label{Lucas}(\cite[Section XXI]{lucas1878theorie})
For non-negative integers m and n and a prime p, we have
\begin{center}
$\tbinom{m}{n} \equiv \prod\limits_{i  = 1}^{k} \tbinom{m_i}{n_i}\mod p$,
\end{center}
where $m = \sm_{i = 0}^{k}m_ip^i, \ n = \sm_{i = 0}^{k}n_ip^i$ are the base p representations of m and n respectively. Here we use the convention $\tbinom{x}{y} = 0,$  if $x < y$. \\ 
\end{theorem}

\begin{corollary}\label{LucasCor1}
Let $m,n$ non-negative integers, then
\begin{enumerate}

    \item 
 $\tbinom{m}{n} \equiv 0\mod2$, for $m = $even and $n = odd$.\\

    \item
$\tbinom{2m + 1}{2n} \equiv \tbinom{m}{n} \mod  2$.

    \item
Consider the base 2 representations $m = \sm_{i = 0}^tm_i2^i, \ n = \sm_{i = 0}^sn_i2^i$. If there exists $r_0 \leq \min\{t,s\}$ such that  $m_{r_0} = n_{r_0} = 1$, then $\tbinom{m+n}{n} \equiv 0\mod 2$.     
\end{enumerate}
\end{corollary}

\section{Main result}
Unless otherwise stated explicitly, for the rest of the paper we are working with a vector space $V$ of dimension $n$ and $\la = (a,b,1^{d})\in \Lnr, \ \mu = (a+d,b)\in \Lambda^{+}(2,r)$, where $b,d \geq 2$, $r \leq n$. \

To describe the generator of $Hom_S(\Delta(\la), \Delta(\mu))$, we will need a specific type of row semistandard tableaux  which we introduce below

\begin{definition}
Let $n,s$ be two integers such that  $0\leq s \leq n$ and $1\leq n$. We recall the set $\Sh (s,n-s)\subseteq G_n$ of $(s,n-s)-$shuffle permutations  defined as 
\begin{align*}
\Sh (s,n-s) = \{ \sigma \in G_n \ | \ \sigma(1) < \cdots < \sigma(s) \ \textit{and} \  \sigma(s+1) < \cdots < \sigma(n)\},    
\end{align*}
where $\Sh(0,n) = \Sh(n,0) = \{1_{G_n}\}$. Furthermore, $|\Sh(s,n-s)| = \binom{n}{s}$.
\end{definition}

\begin{definition}\label{Maintableaux}
Consider the $d-$tuple $(x_1 \dts x_d) = (3 \dts d+2)$. Then for each $h \leq \min\{b,d\}$ and $\sigma \in \Sh(d-h,h)\subseteq G_d$, we define the row semistandard tableau $T_{h,\sigma} \in RSST_{\la}(\mu)$ as follows
\begin{align*}
T_{h,\sigma} = \begin{bmatrix*}[l]
1^{(a)}  2^{(h)} x_{\sigma(1)} \cdots x_{\sigma(d-h)}\\
2^{(b-h)} x_{\sigma(d-h+1)} \cdots x_{\sigma(d)} 
\end{bmatrix*}, 
\end{align*}    
and the element $F_h = \sm_{\sigma\in \Sh(d-h,h)}\phi_{T_{h,\sigma}} \in \Hom_S(D(\la), \Delta(\mu))$. Moreover,
\begin{align*}
T_{d,1_{G_d}} =\begin{bmatrix*}[l]
1^{(a)}  2^{(d)} \\
2^{(b-d)} 3 \cdots d+2 
\end{bmatrix*}, \   T_{0,1_{G_d}} =\begin{bmatrix*}[l]
1^{(a)}   3 \cdots d+2 \\
2^{(b)} 
\end{bmatrix*}  
\end{align*}
\end{definition}
In other words, a tableau $T_{h,\sigma}$ has a stable part associated with the integers $a,b,h$, and the rest of the tableau is a $(d-h,h)-$shuffle permutation of the set $\{3 \dts d\}$.

\begin{example}
For $d = 3$, $a=2$, $b=4$, $h=2$ we have $(x_1,x_2,x_3) = (3,4,5)$ and  we have that the set $\{T_{h,\sigma} | \sigma\in \Sh(1,2) \subseteq G_3\}$ is equal to
\begin{align*}
&\left\{ \begin{bmatrix*}[l]
1^{(2)}  2^{(2)}  3 4\\
2^{(2)} 5 
\end{bmatrix*}, \begin{bmatrix*}[l]
1^{(2)}  2^{(2)}  35  \\
2^{(2)} 4 
\end{bmatrix*}, \begin{bmatrix*}[l]
1^{(2)} 2^{(2)}  45  \\
2^{(2)} 3 
\end{bmatrix*}\right\}          
\end{align*}
\end{example}
Now, as  the integer $h$ ranges over the interval $[0,\min\{b,d\}]$ only some of elements $F_1 \dts F_{\min\{b,d\}}$ are needed to describe the generator. Those elements are characterized through  the base 2 representation of the integer $h$, which will be subject of the  Definitions \ref{2-complement}, \ref{MainDefinition}.   

\begin{definition}\label{2-complement}
Let $i,k \in \mathbb{N}$. Consider the base 2 representation of $i  = \sm_{\la = 0}^mi_{\la}2^{\la}$ ($i_m = 1$). We say that $k$ is a 2-complement of $i$, if $i+k$ has a base 2 representation of the form $i + k = \sm_{\la \geq m+1}t_{\la}2^{\la} + \sm_{\la = 0}^m2^{\la}$ ($t_{\la} \in \{0,1\})$. Additionally, in the case where  $i=0$, $k =$odd we will have the convention that $k$ is 2-complement of $i$. 
\end{definition}

 In other words, a 2-complement of an integer $i$, is an integer that completes the missing 2-digits below it's maximum 2-digit. 

\begin{example}
For $i = 2^3 + 1$, we have that every $k$ of the form  $\sm_{\la \geq 4}k_{\la}2^{\la} + 2^2 +2$ ($k_{\la} \in \{0,1\}$), is a 2-complement of $i$.  However, $k= 2^2+1$   is not a 2-complement of $i$, since $i + k =  2^3 + 2^2 + 2$, which missing $2^0$. Moreover, every $k \in \{1, 2,2^5+2\}$ is not a 2-complement of $i$.   
\end{example}

When we say an integer $x$ contains the base 2 representation of an integer $y$, we mean that considering their base 2 representations, say $x = \sm_{\la \geq 0}x_{\la}2^{\la}, \ y =\sm_{\la \geq 0} y_{\la}2^{\la}$, then $y_{\la} = 1$ implies $x_{\la} = 1$, $ \forall \la \geq 0$.

\begin{definition}\label{MainDefinition}
Consider the partitions $\la = (a,b,1^d) \in \Lambda^{+}(n,r), \ \mu = (a+d, b) \in \Lambda^{+}(2,r)$, where $b, d \geq 2$.
\begin{enumerate}
    \item[$i)$]
Assume that $a,b$ are  even integers  and hence $a-b = 2k$. Consider all the odd integers in the interval $[1,\min\{b,d\}]$, say $r_0 \dts r_{l}$ ($r_i = 2i+1$) and define  $i_0$ to be the maximum integer in the interval $[0,l]$ such that $k$ is a 2-complement of $i_0$. Define the set $\Hlm := \{r_h \in \mathbb{N}: h\leq l$ and $h$ contains the base 2 representation of $i_0$\}.   

    \item[$ii)$] 
Assume  $a$ even, $b$ odd and set $\gamma:= a-b$. Define $j_0$ to  be the maximum integer in the interval $[0,\min\{b,d\}]$ such that $\gamma$ is a 2-complement of $j_0$.Define the set $\Hlmm  := \{h \in \mathbb{N}: h\leq \min\{b,d\}$ and $h$ contains the base 2 representation of $j_0$\}.

\end{enumerate}
\end{definition}

 In order to determine $i_0,j_0$ or even the sets $\Hlm, \ \Hlmm$ the following remark proves useful
\begin{remark}\label{Remark-inequality}
\begin{enumerate}
    \item 
$\sm_{i =0}^n2^i< 2^m$, for $n<m$. (This is handy for comparing two integers using their base 2 representations).
    \item 
Let $k,n\in \mathbb{N}$. Then there is at most one integer $y$ in the interval $[2^n,2^{n+1}]$ such that $k$ is a 2-complement of $y$.  

    \item

Any integer of the form $2^n + \sm_{\la = 0}^{n-1}a_{\la}2^{\la}$ ($a_{\la} \in \{0,1\}$) does not constitute a 2-complement of any integer in the interval $[2^n,2^{n+1})$.
\end{enumerate}    
\end{remark}

\begin{example}
Let $a$ even, $b$ odd and $d\geq2$, such that $\gamma = a-b = 2^5+2^3+2+1$ and $l:=\min\{b,d\} = 2^5+ 2^4 +2^3$ (for example, $a = 2^7+ 2^5 + 2^3 + 2^2, \ b = 2^7 + 1 , \ d = 2^5 + 2^4 + 2^3$). \ 

Using Remark \ref{Remark-inequality} 2) we obtain that $j_0 = 2^4 + 2^2$ (where by Remark \ref{Remark-inequality} 1),3) the next integer that $\gamma$ constitutes a 2-complement is $2^6 + 2^4 + 2^2  > l$). We present the elements of $\Hlmm = \{j_0 < \dots < j_{11}\}$ in increasing order \ 

$\begin{array}{lll}
j_0 = 2^4 + \rule{4mm}{.3pt} +2^2 + \rule{4mm}{.3pt}+ \rule{4mm}{.3pt}    &j_6 = 2^4 + 2^3 +2^2 +2+  \rule{4mm}{.3pt}         \\
j_1 = 2^4 +  \rule{4mm}{.3pt}+2^2+  \rule{4mm}{.3pt} +1       &j_7 = 2^4 + 2^3 +2^2 + 2 +1           \\     
j_2 = 2^4 +  \rule{4mm}{.3pt} +2^2 + 2 + \rule{4mm}{.3pt}     &j_8 = 2^5 + 2^4 +  \rule{4mm}{.3pt}+2^2 +  \rule{4mm}{.3pt} + \rule{4mm}{.3pt}   \\                         
j_3 = 2^4 +  \rule{4mm}{.3pt} +2^2 + 2 +1       &j_9 =  2^5+ 2^4 +  \rule{4mm}{.3pt} +2^2 + \rule{4mm}{.3pt}+1          \\                    
j_4 = 2^4 + 2^3 + 2^2+ \rule{4mm}{.3pt} + \rule{4mm}{.3pt}   &j_{10} =  2^5+ 2^4 +  \rule{4mm}{.3pt} +2^2 +2+ \rule{4mm}{.3pt}        \\
j_5 = 2^4 + 2^3 +2^2 + \rule{4mm}{.3pt} +1   &j_{11} =  2^5+ 2^4 +  \rule{4mm}{.3pt} +2^2 +2+1                    
\end{array}$, \ 

\noindent where the symbol $"\rule{4mm}{.3pt}"$ indicate  the missing 2-digit. Also observe that the next integer  that contains the  base 2 of $i_0$ is $2^5+ 2^4 + 2^3 +2^2 > l$. 
\end{example}

The main result  of the paper  is the following

\begin{theorem}\label{Mainresult}
Let $K$ be an infinite field of characteristic $p>0$ and consider the partitions $\la = (a,b,1^d) \in \Lambda^{+}(n,r), \ \mu = (a+d, b) \in \Lambda^{+}(2,r)$, where $b,d \geq 2$. Then the following hold

\begin{enumerate}

    \item[$1)$]
If $chK = 2$, then
\begin{enumerate}

    \item[$i)$]
$\dim_K\Hom_S(\D(\la),\D(\mu)) \leq 1$
    
    \item[$ii)$]
$\Hom_S(\D(\la),\D(\mu)) \neq 0$ if and only if $a$ is even integer.
    
    \item[$iii)$] 
\begin{enumerate}
    \item[$a)$] 
In case that $a,b$ are even, then the element $\sm_{r_h\in \Hlm}F_{r_h}$ induces a generator for the space $\Hom_S(\D(\la),\D(\mu))$.
    \item[$b)$] 
In case that $a$ is even and $b$ is odd, then the element $\sm_{h\in \Hlmm}F_h$ induces a generator for the space $\Hom_S(\D(\la),\D(\mu))$.    
\end{enumerate}
\end{enumerate}    
     
    \item[$2)$] 
If $chK = p >2$, then $\Hom_S(\D(\la),\D(\mu)) = 0$.
\end{enumerate}
\end{theorem}

We will see that Theorem \ref{Mainresult} follows from Corollary \ref{chk=p}, Corollary \ref{Hom,a-odd}, Proposition \ref{CharacterSystem}, Proposition \ref{EvenNumberof1} and Proposition \ref{odd-CharacterSystem}.

\begin{example}
Let $a=2^6 +2^3, \ b=2^6 + 2 +1$ and $d=2^4 +2^2 +1$. \ 

Since $a$ is $even$ and $b$ is $odd$, from Theorem \ref{Mainresult} we obtain that \ 

$dim_KHom_S(\Delta(\la),\Delta(\mu)) = 1$. We will also give a description of the generator. Initially, $\gamma  = a-b = 2^2 +1$ and $\min\{b,d\} = 2^4 +2^2 +1$ (see Remark \ref{Remark-inequality} 1)).  \ 

The maximum integer $j_0$ in the interval $[0,2^4+2^2+1]$ such that $\gamma$ is a 2-complement of $j_0$ is $j_0 = 2^3 +2$ (notice  that from Remark \ref{Remark-inequality} 2),  the next integer for which $\gamma$ constitutes a 2-complement is $2^4+2^3+2 > \min\{b,d\}$). Now we present the elements of $\Hlmm = \{j_0 <  j_1 < j_2 < j_3\}$ 
$\begin{array}{l}
j_0 = 2^3 + \rule{4mm}{.3pt} +2 + \rule{4mm}{.3pt}             \\
j_1 = 2^3 +  \rule{4mm}{.3pt} +2+ 1              \\     
j_2 = 2^3 +  2^2  + 2 + \rule{4mm}{.3pt}        \\                         
j_3 = 2^3 +  2^2 + 2 +1               \\                    
\end{array}$ \ 

Also observe that the next integer  that contains the  base 2 representation of $j_0$ is $2^4+ 2^3 + 2 > \min\{b,d\}$. Consequently, the  element below 
\begin{align*}
\sm_{k=0}^3\sm_{\sigma \in \Sh(2^4+2^3+2-j_k,j_k)}\phi_{T_{j_k, \sigma}} \in Hom_S(D(\la),\Delta(\mu))   
\end{align*}
induces a generetor for $Hom_S(\Delta(\la),\Delta(\mu))$.    
\end{example}

\begin{example}
Let $a=2^7 +2^6, \ b=2^7 + 2^4 + 2^2$ and $d=2^5 +2^4 +2^2$.  \ 

Since $a,b$ are $even$ integers, from Theorem \ref{Mainresult} we obtain that \ 

$dim_KHom_S(\Delta(\la),\Delta(\mu)) = 1$. We will also give a description of the generator. We start with the following computation
\begin{align*}
\gamma  = a-b &= (2^6-2^2)-2^4 = (2^5 +2^4 +2^3 +2^2) -2^4 \\
              &= 2^5+2^3 +2^2 = 2(2^4 +2^2 +2),
\end{align*}
and set $k = 2^4 +2^2 +2$. Furthermore, from Remark \ref{Remark-inequality} 1), we have that $\min\{b,d\} = 2^5 +2^4 +2^2$. \ 

Consider all the odd integers in the interval $[1,\min\{b,d\}]$, say $r_0 \dts r_l$ ($r_i = 2i+1$). Since $\min\{b,d\}$ is even number, then $l= \frac{\min\{b,d\}}{2} = 2^4 + 2^3 + 2$.
From Remark \ref{Remark-inequality} 2),3), we may compute that the maximum integer $i_0$ in the interval $[0,2^4+2^3+2]$ such that $k$ is a 2-complement of $i_0$ is $i_0 = 2^3+1$ and also       the next integer for which $k$ constitutes a 2-complement is larger than $2^5> l$. Now we present the elements of $\Hlm = \{i_0 <  i_1 < i_2 < i_3 < i_4 \}$, \ 

$\begin{array}{ll}
i_0 = 2^3 + \rule{4mm}{.3pt} + \rule{4mm}{.3pt} +1 ,     & i_3 = 2^3 +  2^2 + 2 +1           \\  
i_1 = 2^3 + \rule{4mm}{.3pt} +2+ 1,      & i_4 = 2^4 +2^3  + \rule{4mm}{.3pt} + \rule{4mm}{.3pt} +1        \\     
i_2 = 2^3 +  2^2 + \rule{4mm}{.3pt} +1.    &      \\                          
\end{array}$ \ 

Also observe that the next integer  that contains the  base 2 of $i_0$ is $2^4+ 2^3 + 2 +1> l$. Consequently, the  element below 
 \begin{align*}
\sm_{k=0}^4\sm_{\sigma \in \Sh(2^5+2^4+2^2-i_k,i_k)}\phi_{T_{r_{i_k }, \sigma}} \in Hom_S(D(\la),\Delta(\mu))   
\end{align*}
induces a generetor for $Hom_S(\Delta(\la),\Delta(\mu))$. \\
\end{example}

In section \ref{Settingthestage} we will see another example (where $d=2$), in which we do not make use of the Theorem \ref{Mainresult}.

\begin{remark}
Here, we provide some general information
\begin{enumerate}
    \item[$i)$]
For the case $d = 1$, we know  from \cite{kulkarni2006ext} that $dim_KHom_S(\Delta(\la),\Delta(\mu))=~1$.

     \item[$ii)$]
The case $b = 1$ follows from \cite[Theorem 3.1]{maliakas2021homomorphisms}.

     \item[$iii)$]
Let $n\geq r$ be positive integers and consider the symmetric group $G_r$ on $r$ symbols. Let $\la,\mu\in \Lambda^{+}(n,r)$ and $\Sp(\la)$ be the corresponding Specht module defined in \cite[Section 6.3]{green2006polynomial} (or equivalently \cite[p.13]{james2006representation}). From \cite[Theorem 3.7]{CL}, we have 
\begin{align*}
\dim_K\Hom_S(\D(\la),\D(\mu)) \leq \dim_K\Hom_{\mathrm{G}_r}(\Sp(\mu),\Sp(\la)).   
\end{align*}
for all $p$ and equality if $p >2$. Hence   Theorem \ref{Mainresult} may be considered as a nonvanishing result for the homomorphisms between Specht modules for $chK = 2$.

\end{enumerate}
\end{remark}

\section{Setting the stage}\label{Settingthestage}

Here we explain the strategy of the proof. We also make some reductions in the involved linear systems.\\ 

For convenience, set $dom\bx_{\la} := \bigoplus_{s=1}^{n-1}\bigoplus_{t=1}^{\lambda_{s+1}}D(\la(s,t)) $. Apply the left exact contravariant functor $\Hom_S(-,\D(\mu))$ on the exact sequence (\ref{pres1}) to obtain
\begin{equation}\label{Relations}
\small 0 \longrightarrow \Hom(\D(\la),\D(\mu)) \xlongrightarrow{\pi^{*}_{\Delta(\la)}} \Hom(D(\la),\D(\mu))  \xlongrightarrow{\bx^{*}_{\la}} \Hom(dom\bx_{\la},\D(\mu)). 
\end{equation}
Hence,  $\Hom_S(\D(\la),\D(\mu)) \simeq \im\pi^{*}_{\Delta(\la)} = \ker\bx^{*}_{\la} \subseteq \Hom_S(D(\la),\D(\mu))$. Using the last relations and Proposition \ref{Base} we will try to explicitly describe the elements of $\Hom_S(\D(\la),\D(\mu))$. Let  $\Phi \in \im\pi_{\D(\la)}^{*}$, then
\begin{enumerate}
    \item 
$\Phi = \sm_{T \in \st}c_T\phi_T$, where  $c_T \in K$.
    \item
$\sm_{T \in \st}c_T(\phi_T\circ \bx_{\la}) = 0$.
\end{enumerate}
Let us analyze condition (2) from above. From the definition of $\bx_{\la}$ (see Theorem \ref{BAWpresentation}), $(2)$  is equivalent  to $ \sm_{T \in \st}c_T(\phi_T\circ \bx_{\la,i,t} ) = 0$, for all $i = 1 \dts d+1, \ t = 1 \dts \la_{i+1}$. From Proposition \ref{Cycle}, the element $y_{i,t} = 1^{(\la_1)} \ots i^{(\la_i + t)} \otimes{(i+1)}^{(\la_{i+1} -t)} \ots n^{(\la_n)}$  is an $S$-generator of $D(\la(i,t))$ and  setting $x_{i,t} = \bx_{\la,i,t}(y_{i,t})  = 1^{(\la_1)} \ots i^{(\la_i)} \otimes i^{(t)}{(i+1)}^{(\la_{i+1} -t)} \ots n^{(\la_n)}$,  condition (2) is equivalent to

\begin{equation}\label{equations}
\sm_{T \in \st}c_T\phi_T(x_{i,t}) = 0, \  \textit{for all} \  i = 1 \dts d+1, \ t = 1 \dts \la_{i+1}.  
\end{equation}

Furthermore, since $\mu_2 \leq \la_1$ we may write the set $ \st$ as follows  \\ 
\begin{equation*}
\begin{aligned}
\{&T=\begin{matrix*}[l]
1^{(\la_1)} 2^{(a_2)} 3^{(a_3)} \sts (d+2)^{(a_{d+2})}  \\
2^{(b_2)}3^{(b_3)} \sts (d+2)^{(b_{d+2})} \end{matrix*} \ : \ 0\leq b_s, a_s,  \ a_s + b_s = \la_s\  (s\geq 2), \\
&\sm_{ i = 2 }^{d+2}b_i = \mu_2, \ \sm_{ i = 2 }^{d+2}a_i = \mu_1 - \la_1  \}. 
\end{aligned}
\end{equation*}
and from  definition of $\la,\mu$ we get
\begin{align}
\st = \{ &T=\begin{matrix*}[l]
1^{(a)} 2^{(b-b_2)} 3^{(1-b_3)} \sts (d+2)^{(1-b_{d+2})}  \nonumber \\
2^{(b_2)}3^{(b_3)} \sts (d+2)^{(b_{d+2})} \end{matrix*} \ : \ 0\leq b_s \leq 1 \ \forall s\geq 3, \label{st} \\ &0 \leq b_2 \leq b, \ \sm_{ i = 2 }^{d+2}b_i = b\}. 
\end{align}

In accordance with the previous expression, we define the set  
\begin{center}
$\Blm := \{(b_2 \dts b_{d+2}): 0 \leq b_s \leq 1 \ \forall s\geq 3, \ 0 \leq b_2 \leq b, \ \sm_{i = 2}^{d+2}b_i = b\}$   
\end{center}
and for each $T \in \st$ as in (\ref{st})  let $c_T = c_{b_2, b_3 \dts b_{d+2}}, \ T= T_{b_2 \dts b_{d+2}}$. For each non-negative integer $h$ such that $h\leq \min\{b,d\}$ it follows that  
\begin{align}\label{F_h}
F_h = \sm_{\sigma\in \Sh(d-h,h)}\phi_{T_{h,\sigma}} = \sm_{b_3 \dts b_{d+2} }\Phi_{T_{b-h,b_3 \dts b_{d+2}}} \in \Hom_S(D(\la), \D(\mu)),    
\end{align}
where the sum ranges over all non negative integers $b_3 \dts b_{d+2}$ such that $ b_s \leq 1 \ \forall s \geq 3$, and $\sm_{i = 3}^{d+2}b_i = h$. 

 In the lemma below  we concisely provide  some preliminary calculations presented  in the proof of Theorem 3.1 in~\cite{maliakas2023homomorphisms}. 

\begin{lemma}\label{MD}
For arbitrary  $\la \in \Lnr, \ \mu \in \Lambda^{+}(2,r)$ and $T  \in \st$ as in (\ref{st}), then

\begin{enumerate}
    
    \item 
Let $i = 1, \ t = 1 \dts \la_{2} =b$. Then  for every $t$ such that $t \leq \min\{b,\mu_1 - \la_1 = d\}$, 
\begin{align*}
\phi_T(x_{1,t}) = &\sum\limits_{k_3 \dts k_{d+2}}(-1)^{k}\tbinom{\la_1 - b_2 +t}{t-k}\prod\limits_{i = 3}^{d+2}\tbinom{b_i + k_i}{k_i} \\ 
&\begin{bmatrix*}[l]
1^{(\la_1 + t)}  2^{(b-b_2 - t + k)} 3^{(1-b_3 - k_3)} \sts (d+2)^{(1-b_{d+2}-k_{d+2})}  \\
2^{(b_2-k)} 3^{(b_3+k_3)} \sts (d+2)^{(b_{d+2} + k_{d+2})} \end{bmatrix*}.    
\end{align*} 
 
 where $k = k_3 + \cdots + k_{d+2}$ and the sum ranges over all non negative integers $k_3 \dts k_{d+2}$ such that $t-(b-b_2) \leq k \leq b_2$ and $0 \leq k_s \leq 1-b_s$, for $s = 3 \dts d+2$. In case,  $t > \min\{b, d\}$, then $\phi_T(x_{1,t}) = 0$.\\

    \item
Let $i \geq 2, \ t = 1 \dts \la_{i+1}$. Then

\begin{align*}
\phi_T(x_{i,t}) = &\sum\limits_{\max\{0,t-b_{i+1}\} \leq j \leq \min\{t,a_{i+1}\}}\tbinom{a_i + j}{j}\tbinom{b_i + t - j}{t-j}\\
&\begin{bmatrix*}[l]
1^{(\la_1)}  2^{(a_2)} \sts i^{(a_i + j)} (i+1)^{(a_{i+1}-j)} \sts (d+2)^{(a_{d+2})}  \\
2^{(b_2)} \sts i^{(b_i+t-j)} (i+1)^{(b_{i+1} - t +j)} \sts (d+2)^{(b_{d+2})} \end{bmatrix*}    
\end{align*}
\noindent where $a_2 = b - b_2$, $a_s = 1 - b_s, \forall s = 3 \dts d+2$.
\end{enumerate}
In both $1),2)$ we make the convention that $\tbinom{y}{x}=0$, if $x>y$.

\end{lemma}

 Below we present an example based on Lemma \ref{MD} which we will use later (see Corollary \ref{chk=p}).

\begin{example}\label{Exampd=2}
Consider the case $d=2\leq b$, i.e. $\la = (a,b,1,1),  \ \mu = (a+2,b)$.  \ 

 Then, the set $\st$ is equal to
\begin{equation*}
\left\{T_1:= \begin{bmatrix*}[l]
1^{(a)}  2^{(2)} \\
2^{(b-2)} 3 4 \end{bmatrix*}, 
T_2 := \begin{bmatrix*}[l]
1^{(a)}  34 \\
2^{(b)} \end{bmatrix*},  
T_3 := \begin{bmatrix*}[l]
1^{(a)}  24 \\
2^{(b-1)}3 \end{bmatrix*}, 
T_4 :=\begin{bmatrix*}[l]
1^{(a)}  23 \\
2^{(b-1)}4 \end{bmatrix*}\right\},    
\end{equation*}
where according to $(\ref{st})$, we have $T_{b-2,1,1} = T_1, \ T_{b,0,0} = T_2, \ T_{b-1,1,0} = T_3, \ T_{b-1,0,1} = T_4$. \ 

 Let $\Phi \in \im\pi_{\D(\la)}^{*} \simeq \Hom(\Delta(\la),\Delta(\mu))$, then according to $(\ref{equations})$ there exists coefficients $\{c_i\}_{i = 1}^4 \subseteq K$ such that $\Phi = \sm_{j=1}^4c_j\phi_{T_j}$, and $\sm_{j = 1}^4c_j\phi_{T_j}(x_{i,t}) = 0$, for all $i = 1 \dts d+1, \  t = 1 \dts \la_{i+1}$. Moreover, by Lemma \ref{MD} the previous equations are equivalent to
\begin{align*}
&\sm_{j = 1}^4c_j\phi_{T_j}(x_{1,1}) = 0,  \sm_{j = 1}^4c_j\phi_{T_j}(x_{1,2}) = 0, \sm_{j = 1}^4c_j\phi_{T_j}(x_{2,1}) = 0,  \sm_{j = 1}^4c_j\phi_{T_j}(x_{3,1}) = 0.
\end{align*}
by using Lemma \ref{MD} we get that
\begin{align*}
&\sm_{j = 1}^4c_j\phi_{T_j}(x_{1,1}) =  ((a-b+3)c_1 - c_4-c_3)\begin{bmatrix*}[l]
1^{(a+1)}  2  \\
2^{(b-2)} 3 4 \end{bmatrix*} +   \\
&  ((a-b+2)c_4-c_2)\begin{bmatrix*}[l]
1^{(a+1)}3  \\
2^{(b-1)}4 \end{bmatrix*} + ((a-b+2)c_3 - c_2)\begin{bmatrix*}[l]
1^{(a+1)}4  \\
2^{(b-1)}3 \end{bmatrix*},
\end{align*}
and by Theorem \ref{BAWbase} the equations $\sm_{j = 1}^4c_j\phi_{T_j}(x_{1,1}) = 0 $ are equivalent to 
\begin{equation}\label{eq(11)}
(a-b+3)c_1 - c_4-c_3 = 0, \ (a-b+2)c_4-c_2 = 0, \ (a-b+2)c_3 - c_2= 0.    
\end{equation}
Similarly we get that equations $\sm_{j = 1}^4c_j\phi_{T_j}(x_{1,2}) = 0 $ are equivalent to 
\begin{equation}\label{eq(12)}
\tbinom{a-b+4}{2}c_1 - (a-b+3)c_3 -(a-b+3)c_4+c_2 = 0.    
\end{equation}
Moreover,
\begin{align}
&\sm_{j = 1}^4c_j\phi_{T_j}(x_{2,1}) = 0 \ \Longleftrightarrow \ (b-1)c_1 + 2c_4 = 0, \ bc_3+c_2 =0\label{eq(21)}\\
&\sm_{j = 1}^4c_j\phi_{T_j}(x_{3,1}) = 0 \ \Longleftrightarrow \ 2c_1 =0,\ 2c_2 = 0, \ c_3+c_4 =0\label{eq(31)}    
\end{align}
In order to continue we need to distinguish cases among the  $\chK = p $ and the parity of $a,b$. \ 

\textbf{Case} $\chK = p >2$. Then by  (\ref{eq(21)}), (\ref{eq(31)})  we obtain $c_1 = c_2 = c_3 = c_4 =0$. Consequently, $\Phi = 0$ and $Hom_S(\Delta(\la),\Delta(\mu)) = 0$. \ 

\textbf{Case} $\chK =  2$, $a$ even, $b$ odd. Then equations (\ref{eq(11)}), (\ref{eq(12)}), (\ref{eq(21)}), (\ref{eq(31)}) are equivalent to $c_2 = c_3 = c_4$ and $\tbinom{a-b+4}{2}c_1 = c_2$. Now depending on the parity of the integer $\tbinom{a-b+4}{2}$, we get that either ($c_2 = c_3 = c_4 = c_1$) or ($c_2 = c_3 = c_4 = 0$ and $c_1 \in K$). Consequently, either $Im\pi_{\D(\la)}^{*} = \spn_K\{\phi_{T_1} + \phi_{T_2} + \phi_{T_3} + \phi_{T_4}\}$ or $ Im\pi_{\D(\la)}^{*} = \spn_K\{\phi_{T_1}\}$. In both cases  we have  $dim_KHom_S(\Delta(\la),\Delta(\mu)) = 1$. Using similar logic, we get that $dim_K\Hom_S(\Delta(\la),\Delta(\mu)) = 1$  for $a,b =$even, and $\Hom_S(\Delta(\la),\Delta(\mu)) = 0$ for $b = odd, a=$arbitrary. \\ 

In case where $a$ is even and $b$ is odd we can also prove that the generators of $Im\pi_{\D(\la)}^*$ are those given in Theorem \ref{Mainresult}. Initially, we have $\min\{b,d\} = 2$ and set $\gamma = a-b$. Define $j_0$ as in Definition \ref{MainDefinition} $ii)$. To   compute $j_0$, we will examine the following cases \ 

\textbf{Case} $\tbinom{a-b+4}{2} = even$. Then $\tbinom{a-b+4}{2} = \tbinom{\gamma+2^2}{2} = 0$. Hence, by Lucas Theorem \ref{Lucas}, $2$ is not a digit in the base 2 representation of $\gamma$. Combining this with the fact that   $\gamma$ is $odd$ implies that  $j_0 = 2$. Thus, $\Hlmm = \{2\}$ and 
\begin{align*}
\sm_{h \in \Hlmm} F_h = \sm_{h\in \Hlmm}\sm_{\sigma\in \Sh(d-h,h)}\phi_{T_{h,\sigma}}= \sm_{\sigma\in \Sh(0,2)}\phi_{T_{2,\sigma}} = \phi_{T_{2,1_{G_2}}} = \phi_{T_1}.   
\end{align*}
\textbf{Case} $\tbinom{a-b+4}{2} = odd$. Then $\tbinom{a-b+4}{2} = \tbinom{\gamma+2^2}{2} = 1$. Hence, by Lucas Theorem \ref{Lucas}, $2$ is a digit in the base 2 representation of $\gamma$. Combining this with the fact that   $\gamma$ is $odd$ implies that  $j_0 = 0$. Thus, $\Hlmm = \{0,1,2\}$. Let $G_2 = \{1,\tau = (12)\}$, then

\begin{align*}
\sm_{h \in \Hlmm} F_h &= \sm_{\sigma\in \Sh(2,0)}\phi_{T_{0,\sigma}} + \sm_{\sigma\in \Sh(1,1)}\phi_{T_{1,\sigma}} + \sm_{\sigma\in \Sh(0,2)}\phi_{T_{2,\sigma}} =  \\
&=\phi_{T_{0,1_{G_2}}} + (\phi_{T_{1,1_{G_2}}} + \phi_{T_{1,\tau}}) + \phi_{T_{2,1_{G_2}}} = \phi_{T_2} + \phi_{T_4} + \phi_{T_3} + \phi_{T_1}.
\end{align*}
\end{example}

\begin{remark}
Some observations coming out of the above example.
\begin{enumerate}
    \item 
Since  Theorem \ref{Mainresult} is not being used, this example verifies Theorem \ref{Mainresult}.    
    \item 
For proof of $Hom_S(\Delta(\la),\Delta(\mu)) =0$ when $\chK = p >0$, we didn't use the equations $\sm_{T\in \st}c_T\phi_T(x_{1,t}) = 0$, and neither will we in the general case.

    \item
For proof of $dim_KHom_S(\Delta(\la),\Delta(\mu)) = 1$ when $\chK =  2$, $a=$ even and $b=$ odd, notice that even though we knew the parity of both $a,b$, it was not enough to compute the parity of $\tbinom{a-b+4}{2}$.  Thus we considered cases. In order to avoid such cases in the sequel,  we will need to examine the base 2 representations of $a,b$. 
    
\end{enumerate}
\end{remark}

\begin{remark}
In the rest of this subsection, we will prove three lemmas involving explicit forms of the equations (\ref{equations}). As a direct consequence of these lemmas, we will also prove the part of the theorem where  $\chK = p>2$.
\end{remark}

\begin{lemma}\label{LemmaEquations1}
Equations (\ref{equations}) for $i = 1, \ t = 1 \dts \la_{2} = b$ are equivalent to the equations
\begin{equation}\label{Equations1}
\sm_{k_3 \dts k_{d+2}}(-1)^k\tbinom{a-b + r + t -k}{t-k}c_{b+k-r, r_3 - k_3 \dts r_{d+2} - k_{d+2}} = 0,    
\end{equation}
for all non negative integers $t, r_3 \dts r_{d+2}$  such that $1 \leq t \leq \min\{b,d\}$ and $0\leq r_s \leq 1 \ \forall s\geq 3$, $t \leq r \leq \min\{b,d\}$. Also, $k = k_3 +\cdots + k_{d+2}$ (respectively $r$) and the sum ranges over all non negative integers $k_3 \dts k_{d+2}$ such that $k_s \leq r_s$, for all $s = 3 \dts d+2$.   
\end{lemma}

\begin{proof}
Let $i = 1, \  t = 1 \dts   \min\{b,d\}$ and $y_t := \sm_{T \in \st}c_T\phi_T(x_{1,t})$ (where $y_t = 0$ for $ t > \min\{b,d\})$, then from Lemma \ref{MD} and (\ref{st}), $y_t$ is equal to 
\begin{align*}
&\sm_{b_2 \dts b_{d+2}} c_{b_2 \dts b_{d+2}}\sm_{k_3 \dts k_{d+2}}(-1)^{k}\tbinom{\la_1 -b_2 +t}{t-k}\prod\limits_{i = 3}^{d+2}\tbinom{b_i + k_i}{k_i}\\
&\begin{bmatrix*}[l]
1^{(\la_1 + t)}  2^{(b-b_2 - t + k)} 3^{(1-b_3 - k_3)} \sts (d+2)^{(1-b_{d+2}-k_{d+2})}  \\
2^{(b_2-k)} 3^{(b_3+k_3)} \sts (d+2)^{(b_{d+2} + k_{d+2})} \end{bmatrix*},
\end{align*}
where $b_2 \dts b_{d+2}$ are subject to (\ref{st}) and $k_3 \dts k_{d+2}$ are as in Lemma \ref{MD} (1). \\ 

 Initially, it is clear that $\prod\limits_{i = 3}^{d+2}\tbinom{b_i + k_i}{k_i} = 1$. Moreover, for each $s \in \{ 3 \dts d+2\}$ set $r_s = b_s + k_s, \ r = r_3 +\cdots + r_{d+2}$ and then define the sets
\begin{enumerate}
    \item 
$I :=$ $\{(b_2 \dts b_{d+2}, k_3 \dts k_{d+2})   :  (b_2 \dts b_{d+2}) \in \Blm$ and $k_3 \dts k_{d+2}$ as in Lemma \ref{MD} (1)\}

    \item
For each $r_3 \dts r_{d+2}$ such that $0\leq r_s \leq 1 \ \forall s\geq 3, \ t \leq r \leq b$ we consider the set \ 
\begin{align*}
I_{r_3 \dts r_{d+2}} = \{ (b_2 \dts b_{d+2}, k_3 \dts k_{d+2}): \ &b_2 = b -r + k, \ b_s = r_s -k_s, \\
&0\leq k_s \leq r_s, \ \forall s\geq 3\}      
\end{align*}

\end{enumerate}
It is clear that we have the disjoint union $I = \bigcup_{r_3 \dts r_{d+2}} I_{r_3 \dts r_{d+2}}$, where $r_3 \dts r_{d+2}$ runs according to (2) from above. Also, notice that $I_{r_3 \dts r_{d+2}} \neq \emptyset$, because it contains $(b-r, r_3 \dts r_{d+2}, 0 \dts 0)$. Hence, 
\begin{align*}
y_t = &\sm_{r_3 \dts r_{d+2}}\sm_{(b_2 \dts b_{d+2}, k_3 \dts k_{d+2})\in I_{r_3 \dts r_{d+2}}}(-1)^{k}\tbinom{\la_1 - b +r -k+ t}{t-k}\\
&c_{b+k-r, r_3-k_3 \dts r_{d+2} - k_{d+2}} \begin{bmatrix*}[l]
1^{(\la_1 + t)}  2^{(r-t)} 3^{(1 - r_3)} \sts (d+2)^{(1-r_{d+2})}  \\
2^{(b-r)} 3^{(r_3)} \sts (d+2)^{r_{d+2})} \end{bmatrix*},
\end{align*}
where the first sum  ranges over all integers $r_3\dts r_{d+2}$ such that $0 \leq r_s \leq 1 \ \forall s \geq 3, \ t\leq r \leq b$ (in fact $r \leq \min\{b,d\}$). One may observe that the variables $b_2,b_3 \dts b_{d+2}$ under the second sum can be omitted. Therefore, $y_t$ is equal to
\begin{align*}
\sm_{r_3 \dts r_{d+2}}&\left(\sm_{k_3 \dts k_{d+2}}(-1)^{k}\tbinom{\la_1 - b +r -k+ t}{t-k}c_{b+k-r, r_3-k_3 \dts r_{d+2} - k_{d+2}}\right) \\
&\begin{bmatrix*}[l]
1^{(\la_1 + t)}  2^{(r-t)} 3^{(1 - r_3)} \sts (d+2)^{(1-r_{d+2})}  \\
2^{(b-r)} 3^{(r_3)} \sts (d+2)^{r_{d+2})} \end{bmatrix*},   
\end{align*}
where the second sum ranges over all integers $k_3 \dts k_{d+2}$ such that $0 \leq k_s \leq r_s, \ \forall s \geq 3$. Since $\mu_2 \leq \la_1$, every tableau on the above expresion is semistandard, and so by 
Theorem \ref{BAWbase}  the equations $y_t = 0$ for $ t = 1 \dts \min\{b,d\}$ are equivalent to the equations (\ref{Equations1}). 
\end{proof}

\begin{lemma}\label{LemmaEquations2}
Equations (\ref{equations}) for $i = 2,\   t = 1 \dts \la_{i+1} = 1$ are equivalent to the equations
\begin{equation}\label{Equations2}
(b-q +1)c_{q, 0,b_4 \dts b_{d+2}} +  qc_{\max\{0,q-1\},1,b_4 \dts b_{d+2}} = 0,   
\end{equation}
for all $q, b_4 \dts b_{d+2}$ such that $q \in \{b-\min\{b,d-1\} \dts b\}, \ 0 \leq b_s \leq 1 \ \forall s \geq 4$,  and $\sm_{i  = 4}^{d+2}b_i  = b -q$. 
\end{lemma}

\begin{proof}
Let $i = 2$, $t = 1$ and $T \in \st$ (as in (\ref{st})), then from Lemma \ref{MD} (2) and definition of $\la, \mu$ we have 
\begin{align*}
\phi_T(x_{2,1}) = \tbinom{b-b_2+1-b_3}{1-b_3}\tbinom{b_2+ b_3}{b_3} \begin{bmatrix*}[l]
	1^{(\la_1)}  2^{(b+1-(b_2+b_3))} 4^{(1- b_4)} \sts  (d+2)^{(1- b_{d+2})}  \\
	2^{(b_2+b_3)}  4^{(b_4)}  \sts (d+2)^{(b_{d+2})} \end{bmatrix*},   
\end{align*}
where we set $\epsilon_{b_2,b_3}= \tbinom{b-b_2+1-b_3}{1-b_3}\tbinom{b_2+ b_3}{b_3}$ and  $\epsilon_{b_2,1} = b_2 + 1, \ \epsilon_{b_2,0} = b-b_2+1$
Furthermore, set $y =  \sm_{T \in \st}c_T\phi_T(x_{2,1})$  and according to (\ref{st}) $y$ is equal to 
\begin{align*}
\sm_{(b_2 \dts b_{d+2}) \in B(\la,\mu)}\epsilon_{b_i,b_{i+1}}c_{b_2 \dts b_{d+2}}\begin{bmatrix*}[l]
1^{(\la_1)}  2^{(b+1-(b_2+b_3))} 4^{(1-b_4)}  \sts (d+2)^{(1- b_{d+2})}  \\
2^{(b_2+b_3)} 4^{(b_4)} \sts (d+2)^{(b_{d+2})} \end{bmatrix*}.     
\end{align*}
 Now  observe that the following disjoint union holds 
\begin{equation}\label{DisUnion}
\Blm = \bigcup\limits_{q = b-\min\{b,d-1\}}^bA_q,   
\end{equation}
where $A_q = \{(b_2 \dts b_{d+2}) \in \Blm : b_2 + b_3 = q\}$. Note also that $A_q \neq \emptyset$. Set $m = b-\min\{b,d-1\}$ and then
\begin{align*}
y &= \sm_{q = m}^b\sm_{(b_2 \dts b_{d+2}) \in A_q}\epsilon_{b_i,b_{i+1}}c_{b_2 \dts b_{d+2}}\begin{bmatrix*}[l]
	1^{(\la_1)}  2^{(b+1 -q)} 4^{(1-b_4)} \sts (d+2)^{(1-b_{d+2})}  \\
	2^{(q)} 4^{(b_4)} \sts (d+2)^{(b_{d+2})} \end{bmatrix*}  \\
& = \sm_{q} \sm_{b_4 \dts b_{d+2}}\left(\sm_{b_2,b_{3}}\epsilon_{b_2,b_3}c_{b_2 \dts b_{d+2}}\right) \begin{bmatrix*}[l]
	1^{(\la_1)}  2^{(b+1 -q)} 4^{(1-b_4)} \sts (d+2)^{(1-b_{d+2})}  \\
	2^{(q)} 4^{(b_4)} \sts (d+2)^{(b_{d+2})} \end{bmatrix*},
\end{align*}
where the second sum ranges all $b_4 \dts b_{d+2}$ such that $0 \leq b_s \leq 1  \ \forall s \in \{4 \dts d+2\}, \  \sm_{k \neq 2, 3}b_k = b-q$, and the third sum ranges all $b_2, b_3$ such that $0 \leq b_2 \leq b, \ 0\leq b_3 \leq 1, \ b_2 + b_3 = q$. Since $\mu_2 \leq \la_1$, every tableau on the above expresion is semistandard, and so by 
Theorem \ref{BAWbase}  the equation $y = 0$  is equivalent to the equations
\begin{equation}\label{equationsLemma4.5}
\sm_{b_2,b_3}\epsilon_{b_2,b_3}c_{b_2 \dts b_{d+2}}= 0,   
\end{equation}
for all $q \in \{b - \min\{b,d-1\} \dts b\}$ and $b_4 \dts b_{d+2}$ such that  $0 \leq b_s \leq 1  \ \forall s \in \{4 \dts d+2\},  \ \sm_{k = 4}b_k = b-q$.  Specifically, from definition of $\epsilon_{b_i,b{i+1}}$, equations (\ref{equationsLemma4.5})  are equivalent to the equations of the statement of the lemma.
\end{proof}

\begin{lemma}\label{LemmaEquations3}
Equations (\ref{equations}) for all $i,t$ such that $3 \leq i \leq d+1, \ t = 1 \dts \la_{i+1} = 1$, depending on $p,d$ are equivalent to the following
\begin{enumerate}
     \item 
($\chK= 2, \  d \geq 2$).    
\begin{equation}\label{Equations3}
c_{b_2, b_3 \dts b_{d+2}} =  c_{b_2, b_3' \dts b_{d+2}'},  
\end{equation}
for all $(b_2, b_3 \dts b_{d+2}), \ (b_2, b_3' \dts b_{d+2}') \in \Blm$.      
    
    \item 
($\chK= p>2, \  d\geq3$).
\begin{equation}
c_{b_2, b_3 \dts b_{d+2}} = 0,  
\end{equation}
for all $(b_2, b_3 \dts b_{d+2})\in \Blm$.   
\end{enumerate}

\end{lemma}

\begin{proof}
Let $i \in \{3 \dts d+1\}$, $t = 1$ and $T \in \st$ (as in (\ref{st})), by setting $y_i =  \sm_{T \in \st}c_T\phi_T(x_{i,1})$ and using similar arguments as in the proof of previous Lemma, we get
\begin{align*}
y_i = &\sm_{(b_2 \dts b_{d+2})\in B(\la,\mu)}\epsilon_{b_i,b_{i+1}}c_{b_2 \dts b_{d+2}}\\
&\begin{bmatrix*}[l]
	1^{(\la_1)}  2^{(b-b_2)} \sts i^{(2- (b_i + b_{i+1}))} (i+2)^{(1-b_{i+2})} \sts (d+2)^{(1- b_{d+2})}  \\
	2^{(b_2)} \sts i^{(b_i+ b_{i+1})} (i+2)^{(b_{i+2})} \sts (d+2)^{(b_{d+2})} \end{bmatrix*},     
\end{align*}
where, $\epsilon_{b_i,b_{i+1}} = \tbinom{2-(b_i + b_{i+1})}{1-b_{i+1}}\tbinom{b_i+ b_{i+1}}{b_{i+1}}$ and $\epsilon_{b_i,b_{i+1}} = 2$, if $b_i = b_{i+1}$ and  $\epsilon_{b_i,b_{i+1}} = 1$, otherwise. \ 

 Now, we may observe that the following disjoint union holds 
\begin{align*}
\Blm = \bigcup\limits_{q = 0}^2A_q,   
\end{align*} 
where $A_q = \{(b_2 \dts b_{d+2}) \in \Blm : b_i + b_{i+1} = q\}$. Note also that $A_q \neq \emptyset$. Indeed, since $b\geq 2$ and $q \in \{0,1,2\}$, we have $0 \leq b-q \leq b$, so for arbitrary $b_i, b_{i+1}$ such that $0\leq b_i, b_{i+1} \leq 1,\ b_i + b_{i+1} = q$, we get that $(b-q, 0 \dts b_i,b_{i+1}\dts 0) \in A_q$. Consequently, 
\begin{align*}
y_i = &\sm_{q = 0}^2\sm_{(b_2 \dts b_{d+2}) \in A_q}\epsilon_{b_i,b_{i+1}}c_{b_2 \dts b_{d+2}}\\
&\begin{bmatrix*}[l]
1^{(\la_1)}  2^{(b-b_2)} \sts i^{(2- q)} (i+2)^{(1-b_{i+2})} \sts (d+2)^{(1- b_{d+2})}  \\
2^{(b_2)} \sts i^{(q)} (i+2)^{(b_{i+2})} \sts (d+2)^{(b_{d+2})} \end{bmatrix*}   
\end{align*}
\begin{align*}
= &\sm_{q = 0}^2\sm_{b_j: j\neq i,i+1}\left(\sm_{b_i,b_{i+1}}\epsilon_{b_i,b_{i+1}}c_{b_2 \dts b_{d+2}}\right) \\
&\begin{bmatrix*}[l]
1^{(\la_1)}  2^{(b-b_2)} \sts i^{(2- q)} (i+2)^{(1-b_{i+2})} \sts (d+2)^{(1- b_{d+2})}  \\
2^{(b_2)} \sts i^{(q)} (i+2)^{(b_{i+2})} \sts (d+2)^{(b_{d+2})} \end{bmatrix*},
\end{align*}
where the second sum ranges all $b_j, \  j \in \{2 \dts d+2\}\setminus \{i, i+1\}$ such that $0 \leq b_s \leq 1  \ \forall s \in \{3 \dts d+2\}\setminus \{i, i+1\}, \ 0 \leq b_2 \leq b, \ \sm_{k \neq i, i+1}b_k = b-q$, and the third sum ranges all $b_i, b_{i+1}$ such that $0 \leq b_i, b_{i+1} \leq 1, \ b_i + b_{i+1} = q$. Since $\mu_2 \leq \la_1$, every tableau on the above expression is semistandard, and so by 
Theorem \ref{BAWbase}  the equations $y_i =  0$ for $i  = 3\dts d+1$ are equivalent to the equations
\begin{equation}\label{eqoflemma}
\sm_{b_i,b_{i+1}}\epsilon_{b_i,b_{i+1}}c_{b_2 \dts b_{d+2}}= 0,   
\end{equation}
for all $(i,q,b_2\dts b_{d+2}) \in C:=\{(i,q,b_2 \dts b_{d+2}): i \in \{3 \dts d+1\}, \ q \in \{0,1,2\}, \ 0 \leq b_s \leq 1  \ \forall s \in \{3 \dts d+2\}\setminus \{i, i+1\}, \ 0 \leq b_2 \leq b, \ \sm_{k \neq i, i+1}b_k = b-q\}$.   Now we distinguish cases, \ 

\textbf{Case} $\chK = 2, \ d\geq 2$. From definition of $\epsilon_{b_i,b_{i+1}}$ and $\chK = 2$, equations (\ref{eqoflemma})  are equivalent to the equations
\begin{equation}
c_{b_2\dts b_{i-1},1 ,0\dts b_{d+2}}   = c_{b_2 \dts b_{i-1}, 0,1\dts b_{d+2}},
\end{equation}
for all $(i,1,b_2\dts b_{d+2}) \in C$. It is easy to check that the last equations are equivalent to the equations of the statement of the lemma. \ 

\textbf{Case} $\chK = p >2, \ d\geq 3$.  From definition of $\epsilon_{b_i,b_{i+1}}$, equations (\ref{eqoflemma})  are equivalent to the equations
\begin{align}
&2c_{b_2\dts b_{i-1},1 ,1\dts b_{d+2}} = 0, \  \forall\   (i,2,b_2\dts b_{d+2}) \in C \  \label{1eqchk>2} \\  
&2c_{b_2 \dts b_{i-1}, 0,0\dts b_{d+2}} = 0, \ \forall \  (i,0,b_2\dts b_{d+2}) \in C\label{2eqchk>2} \\ 
&c_{b_2\dts b_{i-1},1 ,0\dts b_{d+2}}   = -c_{b_2 \dts b_{i-1}, 0,1\dts b_{d+2}}, \ \forall \  (i,1,b_2\dts b_{d+2}) \in C.   \label{3eqchk>2}
\end{align}
Since $\chK = p >2$, then the equations (\ref{1eqchk>2}), (\ref{2eqchk>2}) are equivalent to 
\begin{equation}\label{4eqchk>2}
c_{b_2, b_3 \dts b_{d+2}}= 0,    
\end{equation}
for all $(b_2, b_3 \dts b_{d+2}) \in B(\la,\mu)$ such that the sequence $(b_3 \dts b_{d+2})$ has a repetition (i.e. exists $i$ such that $b_i = b_{i+1}$). Furthermore, equations (\ref{3eqchk>2}) implies that 
\begin{equation}\label{5eqchk>2}
c_{b_2, b_3 \dts b_{d+2}} = \pm c_{b_2,b_3' \dts b_{d+2}'}   
\end{equation}
for all $(b_2,b_3 \dts b_{d+2}), (b_2,b_3'\dts b_{d+2}') \in B(\la,\mu)$. Finally, from (\ref{4eqchk>2}), (\ref{5eqchk>2}) and the fact that $d \geq 3$ we get that $c_{b_2 \dts b_{d+2}} = 0$ for all $(b_2 \dts b_{d_2}) \in B(\la,\mu)$.
\end{proof}

\begin{corollary}\label{chk=p}
If $chK = p >2$, then $\Hom_S(\D(\la),\D(\mu)) = 0$.  
\end{corollary} 

\begin{proof}
It is a direct consequence of Example \ref{Exampd=2} and Lemma \ref{LemmaEquations3}.
    
\end{proof}

\section{The case $\chK =2, a= $ \textit{odd}}

 From equations (\ref{Equations1}), (\ref{Equations2}), (\ref{Equations3}) we determine quickly   $\Hom_S(\D(\la), \D(\mu))$ when $a$ is odd and $b$ arbitrary.

\begin{corollary}\label{Hom,a-odd}
If $a$ is odd, then $\Hom_S(\D(\la), \D(\mu)) = 0$.
\end{corollary}

\begin{proof}
Lets start by assuming that $b$ is odd. The case where b is even is slightly different, but it follows a similar line of reasoning.  \

From (\ref{Relations}) it suffices to prove $\im\pi_{\D(\la)}^{*} = 0 $. Consider $\Phi \in \im\pi_{\D(\la)}^{*}$. Based on (1), (2) under the  (\ref{Relations}) we have  $\Phi = \sm_{(b_2 \dts b_{d+2}) \in \Blm}c_{b_2 \dts b_{d+2}}\phi_{T_{b_2 \dts b_{d+2}}}$, following by the equations  $\sm_{T \in \st}c_T\phi_T(x_{i,t}) = 0, \  \textit{for all} \  i = 1 \dts d~+~1, \ t = 1 \dts \la_{i+1}$.  \\
We start by examining the equations $\sm_{T \in \st}c_T\phi_T(x_{1,1}) = 0$, which by  Lemma \ref{LemmaEquations1} are equivalent to
\begin{equation}\label{Eqodd1}
\sm_{k_3 \dts k_{d+2}}(-1)^k\tbinom{a-b + r + 1 -k}{1-k}c_{b+k-r, r_3 - k_3 \dts r_{d+2} - k_{d+2}} = 0,\end{equation}
for all non negative integers $ r_3 \dts r_{d+2}$  such that  $r_s \leq 1 \ \forall s\geq 3$, $ 1 \leq r \leq b$. Also the sum ranges over all non negative integers $k_3 \dts k_{d+2}$ such that $k_s \leq r_s \ \forall s \geq 3$ and $k \leq 1$. Moreover,  Lemma \ref{LemmaEquations3} implies that $ (\ref{Eqodd1})$ is equivalent to
\begin{align*} 
&(a-b+r+1)c_{b-r, r_3 \dts r_{d+2}} + \sm_{k_3 \dts k_{d+2}: \ k = 1}c_{b+k-r, r_3 - k_3 \dts r_{d+2} - k_{d+2}} = 0    \\
\ \Longleftrightarrow\  &(r+1) c_{b-r, r_3 \dts r_{d+2}} + \tbinom{r}{1} c_{b+1 - r, r_3-1, r_4 \dts r_{d+2}}  = 0,
\end{align*}
where $\tbinom{r}{1} =  \left|\{(k_3 \dts k_{d+2}): 0 \leq k_s \leq r_s \ \forall s \geq 3, \ k =1 \}\right|$. As a consequence, we get
\begin{enumerate}
    
    \item
$c_{b-r, r_3 \dts r_{d+2}} = 0$, for all $r_3 \dts r_{d+2}$, such that $0 \leq r_s \leq 1 \ \forall s \geq 3, \ 1 \leq r \leq b, \ r = even$.

    \item
$c_{b, 0 \dts 0} = 0$. 
\end{enumerate}
Note that (2) follows from the last equivalence by choosing $r_3= 1, r_4 = \cdots = r_{d+2} = 0$.

From $(1),(2)$ we deduce $c_{b_2 \dts b_{d+2}} = 0$ for all $(b_2 \dts b_{d+2}) \in \Blm$, such that $b_2 = $odd. Indeed,  let $(b_2' \dts b_{d+2}') \in B(\la,\mu)$, with $b_2' =$odd.  Set $r_s := b_s' \ \forall s \geq 3$, where $r = \sm_{s \geq 3}r_s = b -b_2' = even \leq b$, we may also assume  $1 \leq r$ (else $c_{b_2' \dts b_{d+2}'} = c_{b, 0 \dts 0}= 0)$. Based on (1)  we get $c_{b_2' \dts b_{d+2}'} =  c_{b-r,r_3 \dts r_{d+2}} = 0$. \ 

It remains to prove $c_{b_2 \dts b_{d+2}} = 0$ for all $(b_2 \dts b_{d+2}) \in \Blm$, such that $b_2 = $even. Let $(b_2' \dts b_{d+2}') \in \Blm$ with $b_2' = even$. From   Lemma \ref{LemmaEquations2} (for t = 1) we have 
\begin{equation}\label{Eqodd2}
c_{q, 0,b_4 \dts b_{d+2}} = c_{q-1,1,b_4 \dts b_{d+2}},   
\end{equation}
for all $q, b_4 \dts b_{d+2}$ such that $q \in \{b-\min\{b,d-1\} \dts b\}, \ q =$odd, $ \ 0 \leq b_s \leq 1 \ \forall s \geq 4$,  and $\sm_{i  = 4}^{d+2}b_i  = b -q$. Furthermore, note that not all of the $b_3' \dts b_{d+2}'$ are zero (or else $b_2' = \sm_{k \geq 2}b_k' = b = $odd, which creates contradiction) so from Lemma \ref{LemmaEquations3}, $c_{b_2'\dts b_{d+2}'} = c_{b_2',1,b_4'' \dts b_{d+2}''} =  c_{(b_2'+1) -1,1,b_4'' \dts b_{d+2}''}$ (where, $(1, b_4'' \dts b_{d+2}'')$ is a rearrangement of $(b_3' \dts b_{d+2}')$). Now observe that $b- \min\{b,d-1\}\leq b_2' + 1 \leq b$ (second inequality, comes from the fact that $b_2' \leq b, \ b_2' = $ even and $b = $odd), $b_2' +1 = odd$  and so from  (\ref{Eqodd2}) $c_{b_2'\dts b_{d+2}'} = c_{b_2'+1,0 \dts b_{d+2}'}$, which is zero from before.
\end{proof}

\section{The case $\chK=2, a, b= $ \textit{even}, Part I}
Unless explicitly mentioned otherwise, we assume that both a and b are even. The case where b is an odd is similar, and detailed explanations will be provided in Section \ref{Section,a=even,b=odd}. The situation becomes more intricate when a is even. Therefore, we divide the proof into two parts and present a concise outline of the first part below.\

According to our assumptions, set $\gamma:= a-b = 2k$ for some $k \in \mathbb{N}$. In this section we will focus on the following
\begin{enumerate}
    
    \item 
We will present a system of equations dependent on the parameter k that characterizes the dimension of the space $\Hom_S(\D(\la), \D(\mu))$ (see Proposition \ref{CharacterSystem}).

    \item 
For a deeper understanding, we will solve the system when k=0 (see example \ref{examplek=0}). Specifically, we will take a first look at how base 2 representations become involved in the entire concept. Additionally, the combinatorial tools employed here will prove useful in the sequel.

    \item
A pivotal combinatorial lemma (refer to Lemma \ref{pivotalLemma}) will be developed here. From this lemma, we will not only establish an upper bound on the dimension of  $\Hom_S(\D(\la), \D(\mu)) \leq 1$, but it will also assist us in precisely computing the dimension in Part II.\\ 
\end{enumerate}

Let us start by reminding the situation we are in   (according to (\ref{Relations}) and Lemmas \ref{LemmaEquations1}, \ref{LemmaEquations2}, \ref{LemmaEquations3})

$\Phi \in \im\pi_{\D(\la)}^{*}$ if and only if the following holds
\begin{enumerate}
    \item[$i)$]
$\Phi = \sm_{(b_2 \dts b_{d+2}) \in B(\la,\mu)}c_{b_2 \dts b_{d+2}}\phi_{T_{b_2 \dts b_{d+2}}}$, for some coefficients  \ 

$\{c_{b_2 \dts b_{d+2}}\}~\subseteq~K$.

    \item[$ii)$]
Coefficients $\{c_{b_2 \dts b_{d+2}}\}$ satisfies the equations (\ref{Equations1}), (\ref{Equations2}), (\ref{Equations3}).\\ 
\end{enumerate}
We intend to analyze $i),ii)$ as much as possible. For this reason, we recall the disjoint union
\begin{equation}\label{DisUnion1}
\Blm = \bigcup\limits_{s = b-\min\{b,d\}}^b\Blm_s,     
\end{equation}
where $\Blm_s = \{ (b_2 \dts b_{d+2}) \in \Blm : b_2 = s\}$ and $\Blm_s \neq \emptyset$.

\begin{lemma}\label{Equations123}
The following are equivalent
\begin{enumerate}
    \item 
Coefficients $ \{c_{b_2 \dts b_{d+2}}: (b_2 \dts b_{d+2})  \in B(\la,\mu)\} \subseteq K$ satisfy equations (\ref{Equations1}), (\ref{Equations2}), (\ref{Equations3}).
    \item

\begin{enumerate}

    \item 
For  $s \in \{b-\min\{b,d\} \dts  b\}$, the set $\{c_{b_2 \dts b_{d+2}}: (b_2 \dts b_{d+2}) \in \Blm_s\} $ is a singleton, and we denote its unique element by $c_{b-s}$. In particular,  $\{c_{b_2 \dts b_{d+2}}: (b_2 \dts b_{d+2}) \in \Blm\} = \{c_{b-s}: s \in \{b-\min\{b,d\} \dts  b\}\}$.

    \item 
$ c_{b_2 \dts b_{d+2}} = 0 $  for all \ $(b_2 \dts b_{d+2}) \in \Blm$ such that $b_2 =$ even.

    \item
For all $t, r$  such that $1 \leq t \leq \min\{b,d\}, \  t \leq r \leq \min\{b,d\}$ and $t =$even, $r=$odd, we have
\begin{equation}\label{Equations1-3}
\sm_{\substack{0 \leq j \leq t \\ j = \textit{even}}}\tbinom{\gamma + r + t - j }{t-j}\tbinom{r}{j}c_{r-j} = 0.      
\end{equation}

\end{enumerate}    
\end{enumerate}
\end{lemma}
\begin{remark}
The reason for using the symbol $c_{b-s}$ instead of $c_s$, is to technically get rid of index $b$ later on (see below proof). Also, $\{r-j: r,j$ as (c) from above\} $\subseteq $ $\{b-s : s \in \{b-\min\{b,d\} \dts  b\}\}$, and so $c_{r-j}$ makes sense. 
\end{remark}

\begin{proof}
We are going to prove $(1) \Longrightarrow (2)$ (the reverse implication is similar). First of all, using (\ref{DisUnion1}), it is clear that equations (\ref{Equations3}) are equivalent to 2)a). Furthermore, by employing the disjoint union mentioned in (\ref{DisUnion}), along with the information that $b$ is an even integer and $\chK = 2$, it is easy to demonstrate that equations (\ref{Equations2}) are equivalent to 2)b). Now let us remind  equations (\ref{Equations1})
\begin{equation}\label{Equations1-2}
\sm_{k_3 \dts k_{d+2}}(-1)^k\tbinom{\gamma + r + t -k}{t-k}c_{b+k-r, r_3 - k_3 \dts r_{d+2} - k_{d+2}} = 0, \end{equation}
for all non negative integers $t, r_3 \dts r_{d+2}$  such that $1 \leq t \leq \min\{b,d\}$ and $0\leq r_s \leq 1 \ \forall s\geq 3$,  $t \leq r \leq \min\{b,d\}$. Also, the sum ranges over all non negative integers $k_3 \dts k_{d+2}$ such that $k_s \leq r_s$, for all $s = 3 \dts d+2$. Using equations (\ref{Equations3}), the left hand of equation (\ref{Equations1-2}) is equal to
\begin{align*}
&\sm_{j = 0}^t\tbinom{\gamma + r + t -j}{t-j}\sm_{\substack{k_3 \dts k_{d+2}: \\ 0\leq k_s \leq r_s, \ k = j}}c_{b+k-r, r_3 -k_3 \dts r_{d+2} - k_{d+2}} = \\
&\sm_{j = 0}^t\tbinom{\gamma + r + t -j}{t-j} \left|\{(k_3 \dts k_{d+2}):  0\leq k_s \leq r_s \ \forall s \geq 3, \ k = j\}\right| c_{r-j} =  \\
&\sm_{j = 0}^t\tbinom{\gamma + r + t -j}{t-j}\tbinom{r}{j}c_{r-j}.
\end{align*}
Now, considering $2)b)$, Corollary \ref{LucasCor1} (1), and $\gamma =$ even, we can deduce $\tbinom{\gamma + r + t -j}{t-j}\tbinom{r}{j}c_{r-j} = 0$, unless $r =$odd and $t,j = $even. Consequently, 2)c) holds.
\end{proof}

 In fact, we can introduce some additional equivalent technical modifications to equations (\ref{Equations1-3}), which will prove to be useful in the process. \ 

Consider all the odd integers in the set $\{1 \dts \min\{b,d\}\}$, say $r_0 \dts r_l$ ($r_i = 2i +1$). Then equations (\ref{Equations1-3}) are equivalent to
\begin{align*}
\sm_{j = 0}^s\tbinom{\gamma + r_i + 2s-2j}{2s-2j}\tbinom{r_i}{2j}c_{r_i - 2j} = 0, \ \forall  s,i \ \textit{such that} \ 1 \leq i \leq l, \ 1 \leq s \leq i. 
\end{align*}
By Corollary \ref{LucasCor1} and the fact that $\gamma = 2k, \ r_i = 2i+1$, equations (\ref{Equations1-3}) are equivalent to:
\begin{equation}\label{Equation1-4}
\sm_{j = 0}^s\tbinom{k + i + s-j}{s-j}\tbinom{i}{j}c_{r_{i -j}} = 0, \ \forall  s,i \ \textit{such that}  \ 1 \leq i \leq l, \ 1 \leq s \leq i.\     
\end{equation}

 For clarity, we introduce the following notation.\ 

\begin{enumerate}

    \item 
$a_k(x,y) := \tbinom{k + y + x}{x}$ and $\epsilon(x,y):= \tbinom{y}{x}$.
    
    \item
Consider all the odd integers in the set $\{1 \dts \min\{b,d\}\}$, say $r_0 \dts r_l$ ($r_i = 2i +1$). Let us define the linear system of equations $E_k := \{E_{i,k}\}_{i = 1}^l$, with corresponding matrices $A_k$ and $\{A_{i,k}\}_{i = 1}^l$ respectively, where  
\begin{align}
&E_{i,k} := \left\{ \sm_{j = 0}^sa_k(s-j,i)\epsilon(j,i)c_{r_{i -j}} = 0 \right\}_{s = 1 }^i, \label{equation-matrix}\\
&A_k = \begin{pmatrix}
A_{l,k}\\
\vdots \\
A_{1,k}
\end{pmatrix} \in M_{\frac{l(l+1)}{2} \times (l+1)}(K). \nonumber  
\end{align}
\end{enumerate}
Let us provide a clearer representation of the equations  $E_{i,k}$ by presenting the corresponding matrix $A_{i,k}$  
\begin{equation}\label{BaseMatrix}
\begin{psmallmatrix}
c_{r_0} & c_{r_1} &  c_{r_2}  &\dots  &c_{r_{i-1}}  &c_{r_i} &c_{r_{i+1}} &\dots &c_{r_l}\\
1 &\epsilon(i-1,i)a_k(1,i) &\epsilon(i-2,i)a_k(2,i) &\dots &\epsilon(1,i)a_k(i-1,i) &a_k(i,i)        & 0  &\dots &0  \\
0 &  \epsilon(i-1,i) &  \epsilon(i-2,i)a_k(1,i)  &\dots  &\epsilon(1,i)a_k(i-2,i)  &a_k(i-1,i)  &0 &\dots &0            \\
0 & 0 &  \epsilon(i-2,i)   & \dots  &\epsilon(1,i)a_k(i-3,i)  &a_k(i-2,i)  &0 &\dots &0  \\
\vdots & \vdots &  0  &\ddots  &\vdots  &\vdots  &0 &\dots &0 \\
\vdots & \vdots & \vdots  &  &\vdots  &\vdots  &0 &\dots &0  \\
0 & 0 &  0  &\dots  &\epsilon(1,i)  &a_k(1,i) &0 &\dots &0
\end{psmallmatrix}.     
\end{equation}
Henceforth, it is important to note that equations $E_{i,k}$ do not affect the variables $c_{r_{i+1}} \dts c_{r_l}$. Consequently, we will exclude the columns associated with variables $c_{r_{i+1}} \dts c_{r_l}$ in the matrix $A_{i,k}$.

 Now let's summarize the topics we have discussed so far in the following proposition 

\begin{proposition}\label{CharacterSystem}
Consider $a,b$ even. Let $\Phi \in \Hom_S(D(\la),\D(\mu))$, then $\Phi \in \im\pi_{\D(\la)}^{*}$ if and only if there exist  coefficients $c_{r_0} \dts c_{r_l} \in K$ such that 
\begin{enumerate}
    \item 
$\Phi = \sm_{i = 0}^lc_{r_i}F_{r_i}$. 
    \item
The coefficients $c_{r_0} \dts c_{r_l} \in K$ satisfies the equations $E_k = \{E_{i,k}\}_{i = 1}^l$.
\end{enumerate}
Since $\im\pi_{\D(\la)}^{*} \simeq \Hom_S(\D(\la),\D(\mu))$, we get that $\dim_K\Hom_S(\D(\la),\D(\mu)) = l+1 - rank(A_k)$.

\end{proposition}

\begin{proof}
Consider $\Phi \in \im\pi_{\D(\la)}^{*}$.  Then $\Phi = \sm_{(b_2 \dts b_{d+2}) \in B(\la,\mu)}c_{b_2 \dts b_{d+2}}\phi_{T_{b_2 \dts b_{d+2}}}$, for some coefficients  $\{c_{b_2 \dts b_{d+2}}\} \subseteq K$ which  satisfies the equations (\ref{Equations1}), (\ref{Equations2}), (\ref{Equations3}). Furthermore, we have the following computation
\begin{align}
\Phi &= \sm_{s =  b-\min\{b,d\}}^b c_{b-s}\left( \sm_{(b_2 \dts b_{d+2}) \in \Blm_s } \phi_{T_{s, b_3 \dts b_{d+2}}}\right)  \label{Oppositecomputation}  \\
&= \sm_{i = 0}^l c_{r_i}\left( \sm_{(b_2 \dts b_{d+2}) \in \Blm_{b-r_i}} \phi_{T_{b-r_i, b_3 \dts b_{d+2}}}\right)  =  \sm_{i = 0}^l c_{r_i}F_{r_i}. \notag
\end{align}

 Notice that the first equality follows from the  disjoint union  $\Blm = \bigcup\limits_{s = b-\min\{b,d\}}^b\Blm_s$ $(*)$  and  Lemma \ref{Equations123} 2)a),  the second equality follows from Lemma \ref{Equations123} 2)b), the fact that $b= even$ and $\{b-s: b-\min\{b,d\} \leq s \leq b, \ s=$odd\} = $\{r_0 \dts r_l\}$ $(**)$, and the third equality follows from $(\ref{F_h})$ Additionally, from Lemma \ref{Equations123} and (\ref{Equation1-4}), the coefficients $c_{r_0} \dts c_{r_l}$, satisfy equations $E_k = \{E_{i,k}\}_{i = 1}^l$.  \ 

For the reverse implication,  let $\Phi \in \Hom_S(D(\la),\Delta(\mu))$ and coefficients $c_{r_0} \dts c_{r_l} \in K$ such that conditions $(1), (2)$ in the statement hold $(***)$.  \ 

Let $(b_2 \dts b_{d+2}) \in \Blm$, using the disjoint union $(*)$, we get $b_2 \in \{b-\min\{b,d\} \dts b\}$. Now  define, $c_{b_2 \dts b_{d+2}} = 0 $  if $b_2 =$ even and $c_{b_2 \dts b_{d+2}} =c_{b-b_2}$ if $b_2 =$ odd (this definition makes sense from $(**)$). Utilizing these definitions combined with the $(\ref{F_h})$ we may obtain similarly to (\ref{Oppositecomputation}) that 
\begin{align*}
\Phi = \sm_{i = 0}^lc_{r_i}F_{r_i} = \sm_{(b_2 \dts b_{d+2}) \in B(\la,\mu)}c_{b_2 \dts b_{d+2}}\phi_{T_{b_2 \dts b_{d+2}}}.    
\end{align*}
Furthermore,  the coefficients $c_{b_2 \dts b_{d+2}}$ satisfy conditions  $2)a),b),c)$ of Lemma \ref{Equations123}  (observe that 2)c) follows from   (\ref{Equation1-4}) and assumptions $(***)$), and so coefficients $c_{b_2 \dts b_{d+2}}$ satisfy the equations (\ref{Equations1}), (\ref{Equations2}), (\ref{Equations3}). Hence, $\Phi \in \im\pi_{\D(\la)}^*$.
\end{proof}

Proposition \ref{CharacterSystem} provides a comprehensive characterization of the space   $\Hom_S(\D(\la), \D(\mu))$, that depends on the solution set of the linear system $E_k$ over a field of characteristic 2. In order to study  such a system, a fruitful approach is to attempt its solution for k = 0 (equivalently, a = b). For this, we will require some preliminary combinatorial tools.

\begin{lemma}\label{Combinatorics1}
Consider $1 \leq \la \in \mathbb{N}$. Then under the assumption $\chK$ $= 2$, we get
\begin{enumerate}
    
    \item 
If $\la = \sm_{i = 0}^n 2^i$ for some $n \geq 0$, then $a_0(t,\la) = 0, \ \epsilon(t,\la) = 1$, for all $t \in \{1 \dts \la\}$.
    
    \item 
    If $\la \neq \sm_{i = 0}^n 2^i$ for all $n \geq 0$, then there exist $\la_0 \in \{1 \dts \la\}$ such that $a_0(\la_0,\la) = 1$.

    \item
If $\la = \sm_{i = 0}^n 2^i$  for some $n \geq 0$, then $\epsilon(s, \la +s) = 0$ for all $0\neq s \in \mathbb{N}$ such that $1 \leq \la + s < 2^{n+1} +\cdots + 2 + 1$.

\end{enumerate}
\end{lemma}

\begin{proof}
All of them can be proved using Theorem \ref{Lucas} and Corollary \ref{LucasCor1} (3).    
\end{proof}

\begin{example}\label{examplek=0}
Assume that $k = 0$ and consider $h \in \{1 \dts l\}$ such that $h = 2^m +\cdots +2 +1$, with $m = $maximum. Then  $\im\pi_{\D(\la)}^* = \spn_K\{F_{r_h}\}$   and  so $\dim_K\Hom_S(\D(\la),\D(\mu)) = 1$.    
\end{example}

\begin{proof}
From Proposition \ref{CharacterSystem} it is sufficient to solve the system $E_0$. Lets start by analyzing equations $E_{h,0}$, according to (\ref{BaseMatrix}) and Lemma \ref{Combinatorics1} we get

\begin{equation}\label{basematrix}
A_{h,0} = \begin{pmatrix}
c_{r_0} & c_{r_1} &  c_{r_2}  &\dots  &c_{r_{h-1}}  &c_{r_h} \\

1       &0 &0 &\dots &0 &0         \\

0 &1    &0  &\dots  &0 &0             \\

0 & 0 &  1   & \dots  &0 &0    \\

\vdots & \vdots &  0  &\ddots  &\vdots  &\vdots   \\

\vdots & \vdots & \vdots  &  &\vdots  &\vdots    \\

0 & 0 &  0  &\dots  &1 &0 
\end{pmatrix}. \\    
\end{equation}
which means $E_{h,0} \ \Longleftrightarrow \ c_{r_0} = \cdots = c_{r_{h-1}} = 0, \ c_{r_h} \in K$. As a consequence, equations $E_{1,0} \dts E_{h-1,0}$ are trivial (means equivalent to 0 = 0). It remains to analyze equations $E_{h+1,0} \dts E_{l,0}$. We will prove by induction to $s$, that $E_{h+s,0} \Longleftrightarrow c_{r_{h+s}} = 0, \ c_{r_h} \in K$, for all $s$ such that $h < h+s \leq l$. \ 
For $s=1$, we present the matrix $A_{h+1,0}$ excluding the zero columns and the columns corresponding to zero variables
\begin{align*}
A_{h+1,0} = &\begin{pmatrix}
&c_{r_h} &c_{r_{h+1}} \\
&\epsilon(1,h+1)a_0(h,h+1) &a_0(h+1,h+1) \\
&\vdots & \vdots \\
&\epsilon(1,h+1)a_0(1,h+1) &a_0(2,h+1)  \\
&\epsilon(1,h+1) &a_0(1,h+1) 
\end{pmatrix} \\
= &\begin{pmatrix}
&c_{r_h} &c_{r_{h+1}} \\
&0 &a_0(h+1,h+1) \\
&\vdots & \vdots \\
&0 &a_0(2,h+1)  \\
&0 &a_0(1,h+1) 
\end{pmatrix}.    
\end{align*}
 
The equality follows from Lemma \ref{Combinatorics1} (3). Furthermore by Lemma \ref{Combinatorics1} (2), the column of variable $c_{r_{h+1}}$ has at least one 1 and so $E_{h+1,0} \Longleftrightarrow c_{r_{h+1}} = 0, \ c_{r_h} \in K$. For the induction step, consider $s \geq 2$ such that $h< h+s \leq l$. Then $E_{h+t,0} \Longleftrightarrow c_{r_{h+t}} = 0, \ c_{r_h} \in K$, for all $t$ such that $h < h+t \leq l$  and $t \leq s-1$. Building on this we present the matrix $A_{h+s,0}$ excluding the zero columns and the columns corresponding to zero variables
$A_{h+s,0} = \begin{pmatrix}
&c_{r_h} &c_{r_{h+s}} \\
&\epsilon(s,h+s)a_0(h,h+s) &a_0(h+s,h+s) \\

&\epsilon(s,h+s)a_0(h-1,h+s) &a_0(h+s-1,h+s)\\

&\vdots & \vdots \\

&\epsilon(s,h+s) &a_0(s,h+s)  \\

&0 &a_0(s-1,h+s) \\ 
&\vdots & \vdots \\
&0 &a_0(1,h+s) 
\end{pmatrix}$.
Now  observe that $1 \leq h+s < 2^{m+1} +\cdots + 2 + 1$ (the strict inequality holds because $h+s \leq l$ and  $h = 2^{m} +\cdots + 2 +1$ with $m =$ maximum such that $h \in \{1 \dts l\}$), thus from Lemma \ref{Combinatorics1} the first column is zero and the second column has at least one 1. Therefore,  $E_{h+s,0} \Longleftrightarrow c_{r_{h+s}} = 0, \ c_{r_h} \in K$.
\end{proof}

 The general case where $k$ is arbitrary, appears to be more demanding. One might view the system 
$E_k$ as a "translation" of system $E_0$, but this analogy is not entirely accurate.
\begin{example}\label{example-for-k-nonzero} We present some examples where $k,l$ are small integers, and so we can solve the system $E_k$  using only the Theorem \ref{Lucas}.
\begin{enumerate}
    \item
For $k = 2^2, \ l = 10,$ we get  $ E_k \ \Longleftrightarrow \ c_{r_3} = c_{r_7}, \ c_{r_i} = 0, \forall i \neq 3,7$.
    \item 
For $k = 2^2 + 2, \ l = 8,$ we get  $ E_k \ \Longleftrightarrow \ c_{r_1} = c_{r_3} = c_{r_5} = c_{r_7}, \ c_{r_i} = 0, \forall i \neq 1,3,5,7$.
     \item 
For $k = 2^3 + 2, \ l = 20,$ we get  $ E_k \ \Longleftrightarrow \ c_{r_5} = c_{r_7} = c_{r_{13}} = c_{r_{15}}, \ c_{r_i} = 0, \forall i \neq 5,7,13,15$.
\end{enumerate}    
\end{example}

One of the important points in the previous example, that we would also like to maintain going forward, is  relation (\ref{basematrix}). However, need some modifications. Lemma \ref{pivotalLemma} is the analogue of relation (\ref{basematrix}).

\begin{lemma}\label{2complement}
Consider $k,i,h \in \mathbb{N}$, then under the assumption $\chK = 2$, we get

\begin{enumerate}
    \item 
If $k = \sm_{\la \geq m+1}k_{\la}2^{\la} + \sm_{\la = 0}^m2^{\la}$, $(k_{\la} \in \{0,1\})$, then $\tbinom{k+t}{t} = 0$, for all $t$ such that $1 \leq t \leq \sm_{\la = 0}^m2^{\la}$.

    \item 
If $h \leq i$ and $k$ is a 2-complement of $i, i-h$, then $\epsilon(h,i) = 1$.

    \item 
If $k$ is not a 2-complement of i, then, there exists $s \in \{1 \dts i\}$ such that $a_k(s,i) = 1$.
\end{enumerate}
\end{lemma}
 For proof of Lemma \ref{2complement} see appendix.

\begin{lemma}\label{pivotalLemma}
Assuming that $k$ is a 2-complement of $i \in \{1 \dts l\}$, then
\begin{align*}
\{E_{j,k}\}_{j =1}^i \ \Longleftrightarrow \ c_{r_0} = \cdots = c_{r_{i-1}} =0, \ c_{r_i} \in K.    
\end{align*}
\end{lemma}

\begin{proof}
Starting with equations $E_{i,k}$ from Lemma \ref{2complement} (1), the corresponding matrix $A_{i,k}$ is of the form
\begin{equation*}
A_{i,k} = \begin{pmatrix}
c_{r_0} & c_{r_1} &  c_{r_2}  &\dots  &c_{r_{i-1}}  &c_{r_i} \\

1       &0 &0 &\dots &0 &0         \\

0 &\epsilon(i-1,i)   &0  &\dots  &0 &0            \\

0 & 0 &  \epsilon(i-2,i)   & \dots  &0 &0    \\

\vdots & \vdots &  0  &\ddots  &\vdots  &\vdots   \\

\vdots & \vdots & \vdots  &  &\vdots  &\vdots    \\

0 & 0 &  0  &\dots  & \epsilon(1,i) &0 
\end{pmatrix}.    
\end{equation*}
Given the aforementioned relation, the implication $(\Longleftarrow)$ is evident. Shifting our attention to the implication $(\Longrightarrow)$, we observe that $E_{i,k}$ enforces $c_{r_0} = 0$ and $c_{r_i}$ ranges through $ K$. Furthermore, in accordance with Lemma \ref{2complement} (2), it follows that $\epsilon(i-j, i) = 1$ (implying $c_{r_j} = 0$) for every $j$ such that $k$ is a 2-complement of $j$.
Now we will prove by induction on $j$ that equations $E_{j,k}$ imply $c_{r_{j}} = 0$ for every $j = 1 \dts i-1$. For j =1 we have $A_{1,k} = \begin{pmatrix}
c_{r_0} & c_{r_1} \\
1   &a_k(1,1)      
\end{pmatrix}$, which means $E_{1,k}  \ \Longleftrightarrow \ c_{r_0} = 0$ and $a_k(1,1)c_{r_1} = 0$. Now, we will examine two cases for $j=1$.   If $k$ is a 2-complement of $1$, then, as discussed earlier, $c_{r_1} = 0$. In the scenario where $k$ is not a 2-complement of $1$, Lemma \ref{2complement} (3) implies $a_k(1,1) = 1$, and once again, we have $c_{r_1} = 0$. For the induction step now, assuming $j\geq 2$ and that equations $E_{j-1,k}$ imply $c_{r_0} = \cdots = c_{r_{j-1}} = 0$, we consider the matrix 
\begin{align*}
A_{j,k} = \begin{pmatrix}
c_{r_j}  \\
a_k(j,j)  \\
\vdots     \\
a_k(1,j)
\end{pmatrix},    
\end{align*}
where we excluded the zero columns and the columns corresponding to zero variables. Similarly as before,  by examining the cases whether $k$ is or is not a 2-complement of $j$, we establish that, $E_{j,k} \ \Longrightarrow c_{r_j} = 0$. 
\end{proof}

\begin{corollary}
If $a,b$ are even, then $\dim_K\Hom_S(\D(\la),\D(\mu)) \leq 1$.
\end{corollary}

\begin{proof}
From Proposition  \ref{CharacterSystem} we know that $\dim_K\Hom_S(\D(\la),\D(\mu)) = l+1 - rank(A_k)$. It suffices to prove $rank(A_k) \geq l$. We accomplish this using induction on
$l$. We may also assume that $k$ is not a 2-complement of $l$, for if it were, then from Lemma \ref{pivotalLemma} and Proposition \ref{CharacterSystem} we obtain $\dim_K\Hom_S(\D(\la), \D(\mu)) = 1$. Let us recall that $A_k = \begin{pmatrix}
A_{l,k}\\
\vdots \\
A_{1,k}
\end{pmatrix}$. For $l = 1$, $A_k = (A_{1,k}) = (1 \ a_k(1,1))$, where $rank(A_k) = 1 \geq l$.  Now, assume
$l \geq 2$. From the induction step, we can find at least $l-1$ linearly independent rows in the matrix $\begin{pmatrix}
A_{l-1,k}\\
\vdots \\
A_{1,k}
\end{pmatrix}$, denoted as  $R_1 \dts R_{l-1}$. In fact, by (\ref{BaseMatrix}), these rows have the form $R_i = \begin{pmatrix}
r_{i,1} & \dots& r_{i,l} &0 &\dots &0
\end{pmatrix}$. It remains to examine the matrix $A_{l,k}$, whose $(l+1)-$column is of the form $\begin{pmatrix}
a_k(l,l)\\
\vdots \\
a_k(1,l)
\end{pmatrix}$ and by Lemma \ref{2complement} (3), there exists $l_0 \in \{1 \dts l\}$, such that $a_k(l_0,l) = 1$. Now consider the row $R_l = (l_0-$row of $A_{l,k})$. It is clear that $R_1 \dts R_{l-1}, R_l$ are linear independent rows in the matrix $A_k$, implying that $rank(A_k) \geq l$.
\end{proof}

\section{The case $\chK=2, a, b= $ \textit{even}, Part II}

This section is dedicated to solving the system of equations $E_k$ for all $k\in \mathbb{N}$ where $a-b= 2k$. By Lemma \ref{pivotalLemma}   the system $E_k = \{E_{j,k}\}_{j = 1}^l$ exhibits a periodic behavior. Specifically, whenever $k$ is a 2-complement of an $i \in \{1 \dts l\}$, the equations $\{E_{j,k}\}_{j = 1}^i$ lead to the vanishing of all variables  $c_{r_0} \dts c_{r_{i-1}}$, except for  $c_{r_i}$. We still need to observe what transpires among the equations $E_{i,k}$ until the subsequent $E_{i',k}$ where $k$ completes the base 2 representation of $i'$. In Example \ref{example-for-k-nonzero}, particularly for the instance with $k = 2^3 + 2$ and $l = 20$, we observe that  $k$ completes the base 2 representation of $i = 1, \ i = 5$ and the next such integer is  $21> l$, which implies that the periodic pattern persists up to $i = 5$. Within the the interval $i=5$ to $i = 20$, the variables that do not vanish are associated with the integers $5 = 2^2+1, \ 7 = 2^2 + 2 +1, \ 13 = 2^3 + 2^2 +1, \ 15 = 2^3 + 2^2 + 2 + 1$.  An initial observation is that these preceding integers encompass all values between $5$ and $20$ that contain the base 2 representation of $5$. 

In this section, all binomial coefficients are calculated within our field of characteristic 2.

\subsection{Construction of the non vanishing sequence $i_0 \dts i_m$ and first consequences.}\

Let  $i_0$  be the maximum integer in set  the  $\{0 \dts l\}$, such that $k$  is 2-complement of $i_0$. 

 Now consider the  sequence of all integers $i_1 \dts i_m \in \{1 \dts l\}$, defined by the following property, $i_j$  is the minimum integer in the set $\{1 \dts l\}$ such that $i_j > i_{j-1}$  and at least one of the  binomials coefficients $\epsilon(i_j- i_{j-1}, i_j) \dts \epsilon(i_j-i_0,i_j)$ is equal to 1.  \\ 

\begin{lemma}\label{Equations-E_k-1}
The equations $E_k$ imply that $c_{r_j} = 0$ for all $j \in \{0 \dts l\}\setminus\{i_0 \dts i_m\}$. 
\end{lemma}

\begin{proof}
Starting with Lemma \ref{pivotalLemma}, the equations $\{E_{j,k}\}_{j=1}^{i_0}$ implies $c_{r_0} = \cdots = c_{r_{i_0-1}} =0$. Moving forward, our objective is to demonstrate that $c_{r_j}= 0$, for all $j$ such that $i_0 < j< i_1$. For this purpose we will prove by induction on $s$ that $c_{r_{i_0+s}} = 0$ for all $s$ such that $i_0 < i_0 +s < i_1$. \ 
For  $s= 1$, set $j = i_0 + 1$ and, in accordance with (\ref{BaseMatrix}), consider the matrix $A_{j,k}$ (excluding the columns corresponding to zero variables)
\begin{align*}
A_{j,k} =  \begin{pmatrix}
&c_{r_{i_0}} &c_{r_j = r_{i_0 +1}} \\
&\epsilon(j-i_0,j)a_k(i_0,j) &a_k(j,j) \\
&\epsilon(j-i_0,j)a_k(i_0-1,j) &a_k(j-1,j)\\
&\vdots & \vdots \\
&\epsilon(j-i_0,j) &a_k(1,j)  \\
\end{pmatrix}.
\end{align*}
Since $i_0 < j < i_1$, the definition of $i_0,i_1$ implies $\epsilon(j-i_0,j)= 0$, and so the column of variable $c_{r_{i_0}}$ is zero. Furthermore, $k$ is not a 2-complement of $j$ (because $j> i_0$), and so by Lemma \ref{2complement} (3) column of $c_{r_j}$ is non zero. Consequently, $  c_{r_{i_0 + 1}} = c_{r_j}  = 0$. Let $s \geq 2$, by inductive step we have $c_{r_{i_0+1}} = \cdots = c_{r_{i_0+s}} = 0$. Set $i_0 + (s+1) = j$, and consider the matrix $A_{j,k}$  (excluding the columns corresponding to zero variables)
\begin{align*}
A_{j,k} =  \begin{pmatrix}
&c_{r_{i_0}} &c_{r_j = r_{i_0 +(s+1)}} \\
&\epsilon(j-i_0,j)a_k(i_0,j) &a_k(j,j) \\
&\vdots & \vdots \\
&\epsilon(j-i_0,j) &a_k(j-i_0,j)  \\
&0 & a_k(j-(i_0+1),j) \\
&\vdots & \vdots \\
&0 &a_k(1,j)
\end{pmatrix}.
\end{align*}
Since $i_0 < j < i_1$, the definition of $i_0,i_1$ imply  $\epsilon(j-i_0,j)= 0$ and Lemma \ref{2complement} (3) forces $c_{r_j} = c_{r_{i_0 + (s+1)}} = 0$.  
Now we will prove that $c_{r_j} = 0$, for all $j,t$ such that $i_t < j < i_{t+1}$, which will be accomplished by induction on $t$. The case $t=1$ has already been proved. Now, let $t \geq 2$. By the induction step, $c_{r_j} = 0$ for all integers $j\in (i_0,i_1)\cup\cdots \cup (i_t,i_{t+1})$. Continuing with  the same line of argument  as in case $t=1$, and utilizing  induction on $s$, we can establish that $c_{r_{i_{t+1} +s}} = 0$, for all $s$ such that  $i_{t+1} < i_{t+1} +s < i_{(t+1) +1}$. In other words, $c_{r_j} = 0$ for all $j$ such that $i_{t+1} < j < i_{(t+1) +1}$. It remains to prove $c_{r_j} = 0$ for all $j$ such that $l\geq j > i_m$. Utilizing the way we defined of $i_1 \dts i_m$ and the fact that $c_{r_0} = \dots  = c_{r_{i_0-1}} = 0$, $c_{r_j} = 0$ for all $j \in (i_0,i_1)\cup\cdots \cup (i_{m-1},i_{m})$ we can prove by induction on $s$ that $c_{r_{i_m+s}} = 0$ for all $s$ such that $i_m < i_m +s \leq l$. 
\end{proof}

\begin{corollary}\label{Equations-E_k-2}
The equations $E_k =\{ E_{i,k}\}_{i =1}^l$ are equivalent to the following equations
\begin{equation}\label{equations-(i_1,...,i_m)}
E_{i_1,k}  \dts  E_{i_m,k} \ \textit{and} \ c_{r_j} = 0 \ \textit{for all} \ j \in \{0 \dts l\}\setminus\{i_0 \dts i_m\}.    
\end{equation}
\end{corollary}

\begin{proof}
The implication $(\Longrightarrow)$ follows from Lemma  \ref{Equations-E_k-1}. For the inverse implication it suffices to prove that equations $\{E_{j,k}\}_{j \neq i_1 \dts i_m}$ can be omitted, i.e., are equivalent to $0=0$. Initially, considering the definition of the sequence $i_0 \dts i_m$, it is clear that $\epsilon(j-i_0,j)= \cdots = \epsilon(j-i_m,j) = 0$, for all $j$ satisfying  $i_0 < j \leq l$ and $j \neq i_1 \dts i_m$. Using the previous observation and the assumption $(\Longleftarrow)$, it is straightforward to verify that for every $j\neq i_0 \dts i_m$, the   matrix $A_{j,k}$ according to (\ref{BaseMatrix}) either has columns corresponding to zero variables or variables corresponding to zero columns. Therefore, equations $E_{j,k}$ can be omitted, for $j \neq i_1 \dts i_m$. The only remaining equations to be omitted are $E_{i_0,k}$. Indeed, according to assumption $(\Longleftarrow)$ and the fact that the column of $c_{r_{i_0}}$ (in the matrix $A_{i_0,k}$) is zero since $k$ is a 2-complement of $i_0$ (see Lemma~\ref{2complement}~(1)), we get that equations $E_{i_0,k}$ can be omitted.
\end{proof}

Our next objective, is to represent the sequence $i_0 \dts i_m$ based on the base 2 representation of its terms. For this purpose, we require the following two lemmas (for proofs see appendix). The first lemma will also prove useful in the sequel.

\begin{lemma}\label{2-digits-1}
Let $b_1 <\cdots<  b_m, \ a_1 <\cdots<  a_n$ be two sequences of non-negative integers, such that $2^{b_m}+\cdots + 2^{b_1}  \geq 2^{a_n} +\cdots + 2^{a_1}$. Consider the element $x = (2^{b_m}+\cdots + 2^{b_1})  - (2^{a_n} +\cdots + 2^{a_1})$, then
\begin{enumerate}
    \item 
The base 2 representation of the element x, contains a 2-digit in the set $\{2^{a_1} \dts 2^{a_n}$, $2^{b_1} \dts 2^{b_m}\}$, if and only if $x > 0$. 
    \item
If $x>0$, then $x$ has base 2 representation of the form  $x = \sm_{\la \geq \min\{a_1,b_1\}}x_{\la}2^{\la}$.
    \item 
If $x>0$ and $b_1 > a_1$, then the base 2 representation of the element x, contains $2^{a_1}$ as a 2-digit.    
\end{enumerate}
\end{lemma}

\begin{lemma}\label{2-digits-2}
Let  $1 \leq t\leq s$, and  consider the base 2 representations, $s = \sm_{\la = 0}^ms_{\la}2^{\la}, \ t = \sm_{\la = 0}^nt_{\la}2^{\la}$, $(t_m = s_n = 1)$. Then under the assumption $\chK = 2$, we get 
\begin{center}
$\epsilon(s-t,s) =1  \Longleftrightarrow $ $s$ contains the base 2 representation of $t$ (i.e. $ t_{\la} \leq s_{\la}$, for all $\la = 1 \dts n$).      
\end{center}
\end{lemma}

\begin{corollary}\label{2-base-i_0,...,i_1}
We can explicitly describe the sequence $i_0 \dts i_m$  as follows
\begin{enumerate}
    \item 
$i_0 = \max\{i \in \{0 \dts l\}  :  k$ is $2-$complement of i\}.
    \item
$\{i_1 <\cdots<   i_m\} = \{i \in \{0 \dts l\}: i_0< i $ and $i$ contains the base 2 representation of $i_0\}$.\\ 
\end{enumerate}
\end{corollary}

\subsection{Solution of the system $E_k$ through sequence $i_0 \dts i_m$.}\ 

In light of Corollary  \ref{Equations-E_k-2}, it is sufficient to focus on the equations $E_{i_1,k} \dts$ $E_{i_m,k}$ or equivalently the corresponding matrices $A_{i_1, k} \dts A_{i_m, k}$. Nevertheless, these matrices continue to appear perplexing.   Employing Lemma \ref{2-digits-2}, there are still simplifications to make. Henceforth, fix $s\in \{1 \dts m\}$ and for the remainder of this subsection, we will study the matrix $A_{i_s,k}$. \

Up to this point, the only potentially non-trivial columns of $A_{i_s,k}$, are those of variables $c_{r_{i_0}} \dts c_{r_{i_s}}$. However, upon considering the integers $i_0 = y_0 <\cdots<  y_t = i_s$,  where 
\begin{equation}
\small\{y_0 \dts y_t\} = \{y \in \{i_0 \dts i_s\}: i_s \ \textit{contains the base 2 representation of} \  y\},    
\end{equation}
the application of Lemma \ref{2-digits-2} simplify the matrix $A_{i_s,k}$ as follows (where we have exclude zero columns and columns corresponding to zero variables)
\begin{equation}\label{finalmatrix}
\begin{pmatrix}
&c_{r_{i_0 = y_0}} &\dots &c_{r_{y_{t-1}}} & c_{r_{y_t = i_s}} \\
&a_k(y_0,y_t) &\dots &a_k(y_{t-1},y_t) &a_k(y_t,y_t) \\
&\vdots &\vdots &\vdots &\vdots\\
&a_k(1,y_t)    &\dots    &a_k(y_{t-1}-(y_0-1),y_t)        &a_k(y_{t}-(y_0-1),y_t)   \\
&1 & \dots &a_k(y_{t-1}-y_0,y_t)  &a_k(y_t-y_0,y_t) \\
&0 &\dots  &a_k(y_{t-1}-(y_0+1),y_t)   &a_k(y_{t}-(y_0+1),y_t)  \\
&\vdots & \vdots &\vdots &\vdots  \\
&0 & \dots & a_k(1,y_t)  &a_k(y_{t}-(y_{t-1}-1),y_t)  \\
&\vdots  &\dots &1 &a_k(y_t-y_{t-1},y_t) \\
&\vdots  &\dots &0 &a_k(y_t-(y_{t-1}+1), y_t)  \\
&0 &\dots &\vdots &\vdots \\ 
&0 &\dots &0 &a_k(1,y_t)
\end{pmatrix}.    
\end{equation}

 The key observation is that each row of the matrix (\ref{finalmatrix}) contains an even number of units.

\begin{proposition}\label{EvenNumberof1} With the above notations we have the following
\begin{enumerate}
    \item 
The matrix (\ref{finalmatrix}) contains an even number of units in each row.

    \item
$E_{i_s,k} \ \Longleftrightarrow \ c_{r_{i_0 = y_0}} = c_{r_{ y_1}} = \cdots =  c_{r_{i_s = y_t}}.$ \end{enumerate}
Since $s \in \{1 \dts m\}$ was arbitrary, we get that equations $E_k $ are equivalent to, $c_{r_{i_0}} = \cdots = c_{r_{i_m}}$ and $c_{r_j} = 0 \ \textit{for all} \ j \in \{0 \dts l\}\setminus\{i_0 \dts i_m\}$.

\end{proposition}

\begin{proof}
The proof of property $(1)$ will be  the subject of discussion for subsections \ref{Subsection R1-Ry0+1}, \ref{subsection y0+2-yt-1+1} and \ref{Subsection Ryt-1+2-Ryt}. Regarding property $(2)$, one may deduce it directly from the following three facts,  \ i) $chK = 2$,  ii) property $(1)$ and iii) we guarantee at least one $1$ in the columns corresponding the variables $c_{r_{y_0}} \dts c_{r_{y_{t-1}}}$. 
\end{proof}

 For the proof of Proposition \ref{EvenNumberof1} (1), the plan goes as follows. First, we will provide the base 2 representations of the elements $y_i$. Then, we will present a convenient computational formula for the binomial coefficients of the matrix \ref{finalmatrix}. Lastly, for the cardinality of each row we will distinguish the cases $(1 \longrightarrow y_0 +1), \ (y_0+2 \longrightarrow  y_{t-1} + 1)$ and $(y_{t-1} + 2  \longrightarrow y_t)$. \\

Recall the set $I(\nu,k)$ of all maps $\{1 \dts k\} \longrightarrow \{1 \dts \nu\}$. 
\begin{lemma}\label{2-base of y_1,...,y_t}
Consider the base 2 representation of $y_0 = i_0 = \sm_{\la \geq 0}d_{\la}2^{\la}$.  Then
\begin{enumerate}

    \item 
(base 2 representation of $y_t = i_s$).  $y_t = 2^{\la_{\nu}} +\cdots + 2^{\la_1} + i_0$, for some non negative integers  $\la_1 <\cdots<  \la_{\nu}$ such that, $d_{\la_j} = 0$ for all $j\in \{1 \dts \nu\}$.

    \item
(base 2 representation of $y_1 \dts y_t$). $\{y_1 \dts y_t\} = \bigcup\limits_{k = 1}^{\nu}\{2^{\la_{I_k}} +\cdots + 2^{\la_{I_1}} +i_0 : I \in I(\nu,k), \  I$ is strictly increasing\}. Notice that the union is disjoint.
\end{enumerate}
\end{lemma}
     
\begin{proof}
Lemma follows directly by Corollary \ref{2-base-i_0,...,i_1} and construction of integers $y_0 \dts y_t$.    
\end{proof}

\begin{lemma}\label{formula}
Regarding the sequence $\la_1 <\cdots<  \la_{\nu}$, for any non negative integer $x$ such that  $1 \leq x \leq y_t$, we have the following formula \ 

$a_k(x,y_t) = \begin{cases}
 0,    &\textit{if} \ x \ \textit{has a 2-digit in } \  \{2^{\la_2} \dts 2^{\la_{\nu}}\}\cup\{2^x: x \leq \la_1-1\}\\
 1,    & \textit{otherwise}
\end{cases}.$
\end{lemma}

\begin{proof}
Lets start by assuming $i_0 \geq 1$ (we will treat case $i_0 =0$ at the end). Consider the base 2 representation $i_0 = \sm_{\la = 0}^nd_{\la}2^{\la}$ $(d_n = 1)$. Since $k$ is a 2-complement of $i_0$ we may write $k+i_0 = \sm_{\la \geq n+1}\gamma_{\la}2^{\la} + \sm_{\la = 0}^n2^{\la}$ and thus $a_k(x,y_t) = \binom{ \sm_{\la \geq n+1}\gamma_{\la}2^{\la} + 2^{\la_{\nu}} +\cdots + 2^{\la_1} + \sm_{\la = 0}^n2^{\la} + x}{x}$. From Lemma \ref{2-base of y_1,...,y_t}, notice that $\la_i \neq n$, for all $i = 1 \dts \nu$. To proceed, we need to distinguish among the following three cases\ 

 \textbf{Case 1}. $\la_1 <\cdots<  \la_{\nu} < n$.  Then,  
\begin{align*}
a_k(x,y_t) = \binom{\sm_{\la \geq n+1}\gamma_{\la}2^{\la} + 2^{n+1} + 2^{\la_{\nu}} +\cdots + 2^{\la_2} + \sm_{\la = 0}^{\la_1-1}2^{\la} + x}{x}.     
\end{align*}
Furthermore, from Lemma \ref{2-base of y_1,...,y_t} and case assumption, we get that $x \leq y_t \leq \sm_{\la = 0}^n2^{\la} < 2^{n+1}$, thus, $x$ has a base 2 representation of the form $x= \sm_{\la =0}^mx_{\la}2^{\la}$, where $m \leq n$, combine this with  Theorem \ref{Lucas} and Corollary \ref{LucasCor1} (3), we may  deduce the required formula. \\ 

 \textbf{Case 2}. There exist $s_0 \in \{1 \dts \nu-1\}$ such that $\la_1 <\cdots<  \la_{s_0} < n < \la_{s_0 +1} <\cdots<  \la_{\nu}$. We start with the following  critical observation, $\gamma_{n+1} = \cdots = \gamma_{\la_{\nu}} = 1$. Indeed, let us assume the opposite; that is, suppose there exists $\tau_{0}$ such that $n+1 \leq \tau_0 \leq \la_{\nu}$ and $\gamma_{\tau_0} = 0$. Moreover, consider $\tau_0$ to be minimum, meaning $\gamma_{\la} = 1$, for every $\la$ such that $n+1 \leq \la < \tau_0$. This lead us to the following computations
\begin{align*}
k + (2^{\tau_0} + i_0) = \sm_{\la \geq n+1}\gamma_{\la}2^{\la} + 2^{\tau_0} + \sm_{\la = 0}^n2^{\la}  = \sm_{\la \geq \tau_0 +1}\gamma_{\la}2^{\la} + \sm_{\la = 0}^{\tau_0}2^{\la}.    
\end{align*}
Additionally, $2^{\tau_0} + i_0 =2^{\tau_0} + \sm_{\la = 0}^nd_{\la}2^{\la}$, where $n < \tau_0$ and so the preceding expression represents the base 2 form of $2^{\tau_0} + i_0$. Therefore, the computations above imply that $k$ is a 2-complement of $2^{\tau_0} + i_0$. On the other hand, $i_0< 2^{\tau_0} + i_0 \leq y_t = i_s \leq l$, which leads to contradiction (this stems from the fact that $i_0$ considered to be the maximum integer in the set $\{0 \dts l\}$ for which $k$ is a 2-complement). Utilizing the property $\gamma_{n+1} = \cdots = \gamma_{\la_{\nu}} = 1$, we obtain 
\begin{align*}
&\sm_{\la \geq n+1}\gamma_{\la}2^{\la} + 2^{\la_{\nu}} +\cdots + 2^{\la_1} + \sm_{\la = 0}^n2^{\la} =\sm_{\la \geq \la_{\nu}+1}\gamma_{\la}2^{\la}+ \sm_{\la = 0}^{\la_{\nu}}2^{\la} + 2^{\la_{\nu}} +\cdots + 2^{\la_1}  \\
&=(\sm_{\la \geq \la_{\nu}+1}\gamma_{\la}2^{\la}+ 2^{\la_{\nu}+1}) + 2^{\la_{\nu}} +\cdots + 2^{\la_2} + \sm_{\la = 0}^{\la_1-1}2^{\la}.    
\end{align*}
Consequently, $a_k(x,y_t) = \binom{\sm_{\la \geq \la_{\nu}+1}\gamma_{\la}2^{\la}+ 2^{\la_{\nu}+1} + 2^{\la_{\nu}} +\cdots + 2^{\la_2} + \sm_{\la = 0}^{\la_1-1}2^{\la} + x}{x}$. Furthermore, from Lemma \ref{2-base of y_1,...,y_t} and case assumption, we get that
\begin{align*}
x &\leq y_t = 2^{\la_{\nu}} +\cdots + 2^{\la_1} + i_0  = (i_0 + 2^{\la_1} +\cdots + 2^{\la_{s_0}}) + 2^{\la_{s_0+1}} +\cdots + 2^{\la_{\nu}} \\
&< 2^{n+1} + 2^{\la_{s_0+1}} +\cdots + 2^{\la_{\nu-1}} + 2^{\la_{\nu}} \leq  2^{\la_{\nu-1}+1} + 2^{\la_{\nu}} \leq 2^{\la_{\nu}} +2^{\la_{\nu}} =  2^{\la_{\nu}+1}.    
\end{align*}
Hence, $x$ has a base 2 representation of the form $x= \sm_{\la =0}^mx_{\la}2^{\la}$, where $m \leq \la_{\nu}$, combine this with  Theorem \ref{Lucas} and Corollary \ref{LucasCor1} (3), we may  deduce the required formula. \ 

 \textbf{Case 3}.  $n< \la_1 <\cdots<  \la_{\nu}$ has similar treatment with Case 2. \ 

Now assume that  $i_0 = 0$ and set $n = 0$. Then $k$ is an odd integer, as it is a 2-complement of $i_0$. In case $\nu = 1$ and $\la_1 = 0$, we get $y_t = 1, \ x \in \{0,1\}$ and  $a_k(1,1) =  a_k(0,1) = 1$, which implies the required formula. In different case there exists $s_0 \in \{1 \dts \nu\}$ such that $n < \la_{s_0}$ and the proof of the required formula is similar to Case 2.
\end{proof}

\subsubsection{Rows $R_1 \dts R_{y_0 + 1}$, contains even number of units.}\label{Subsection R1-Ry0+1}

The following elementary lemma (for proof see appendix) constitutes the heart of the proof that all of the rows of the matrix (\ref{finalmatrix}) contains even number of units.

\begin{lemma}\label{Elementary1}
Let $x\in \mathbb{N}$ and consider a sequence $2^{a_1} <\cdots<  2^{a_k}$, such that some of the terms of this sequence serve as  2-digits of $x$. Then, there exists a unique sequence $b_1 <\cdots<  b_s$, such that $b_i \in \{a_1 \dts a_k\}$ and $x + 2^{b_s} +\cdots + 2^{b_1}$ does not have any 2-digit inside the set $\{2^{a_1} \dts 2^{a_k}\}$.
\end{lemma}

\begin{corollary}
The rows $R_1 \dts R_{y_0 +1}$ of   matrix (\ref{finalmatrix}), contains even number of units.
\end{corollary}

\begin{proof}
Consider $h\in \{1 \dts y_0 +1\}$, $R_h = (a_k(y_0-(h-1),y_t) \dts a_k(y_t - (h-1),y_t))$ and set $x:= y_0 -(h-1) \geq 0$. According to Lemma \ref{2-base of y_1,...,y_t} we get that 
\begin{equation*}
R_h = (a_k(x,y_t), a_k(x+2^{\la_{I^1_{k_1}}} +\cdots +2^{\la_{I^1_1}} ,y_t),...,a_k(x+2^{\la_{I^t_{k_t}}} +\cdots +2^{\la_{I^t_1}},y_t)),   
\end{equation*}
where $\{I^1 \dts I^t\}  = \bigcup\limits_{k = 1}^{\nu}\{I \in I(\nu,k):   I$ is strictly increasing\}. 
We may also assume that $x$ does not have a 2-digit $\leq 2^{\la_1-1}$, otherwise from Lemma \ref{formula},  $R_h = (0 \dts 0)$, and the desired result follows immediately. Now there are some cases we need to distinguish.\ 

 \textbf{Case 1}. $x$ has a 2-digit in the set $\{2^{\la_2} \dts 2^{\la_{\nu}}\}$. \ 

  Then from Lemma \ref{Elementary1} there exists a unique sequence $b_1 <\cdots<  b_s$, such that $b_i \in \{\la_2 \dts \la_{\nu}\}$ and $x + 2^{b_s} +\cdots + 2^{b_1}$ does not have a 2-digit inside the set $\{2^{\la_2} \dts 2^{\la_{\nu}}\}$. Particularly, from Lemma \ref{formula}, $a_k(x+ 2^{b_s} +\cdots + 2^{b_1},y_t) = 1$.

\textbf{(Case 1a)}. $x + 2^{\la_1}$ does not have a 2-digit in the set $\{2^{\la_2} \dts 2^{\la_{\nu}}\}$.  \ 

  Lemma \ref{formula} implies that $a_k(x+2^{\la_1},y_t) = 1$. Now notice that $2^{\la_1} \neq 2^{b_s} +\cdots + 2^{b_1}$ and so we have at least two units in the row $R_h$. We claim that we have exactly two units. Indeed, consider a sequence  $2^{c_1} <\cdots<  2^{c_r}$, such that $c_i \in \{\la_1 \dts \la_{\nu}\}$ and $a_k(x+2^{c_r} +\cdots + 2^{c_1},y_t) = 1$. Then, $x + 2^{c_r} +\cdots + 2^{c_1}$ does not have a 2-digit in the set $\{2^{\la_2} \dts 2^{\la_{\nu}}\}$. In case $c_1  \neq  \la_1$, then uniqueness in Lemma \ref{Elementary1} forces $x+2^{c_r} +\cdots + 2^{c_1} = x+2^{b_s} +\cdots + 2^{b_1}$. In case, $c_1 = \la_1$, we claim that $r = 1$, otherwise from (Case 1a) assumption, the elements $2^{c_2} \dts 2^{c_r} \in \{2^{\la_2} \dts 2^{\la_{\nu}}\}$ are 2-digits of the element $x + 2^{c_r} +\cdots + 2^{c_1} = (x + 2^{\la_1}) + 2^{c_r} +\cdots + 2^{c_2}$, which leads to contradiction. \ 

  \textbf{(Case 1b)}. $x + 2^{\la_1}$  have a 2-digit in the set $\{2^{\la_2} \dts 2^{\la_{\nu}}\}$. \ 

 Lemma \ref{Elementary1} implies that there exists a unique sequence $a_1 <\cdots<  a_r$, such that $a_i \in \{\la_2 \dts \la_{\nu}\}$ and $x + 2^{\la_1} + 2^{a_r} +\cdots + 2^{a_1}$ does not have a 2-digit inside the set $\{2^{\la_2} \dts 2^{\la_{\nu}}\}$.  Particularly, from Lemma \ref{formula}, $a_k(x+ 2^{\la_1} + 2^{a_r} +\cdots + 2^{a_1},y_t) = 1$. Now notice that $2^{\la_1} + 2^{a_r} +\cdots + 2^{a_1} \neq 2^{b_s} +\cdots + 2^{b_1}$, so we have at least two units in the row $R_h$. Similarly to the (Case 1a),  we can prove that we have exactly two units.\

 \textbf{Case 2}. $x$ does not have a 2-digit in the set $\{2^{\la_2} \dts 2^{\la_{\nu}}\}$. \  This case is easier with similar treatment as in Case 1.
\end{proof}

\subsubsection{Rows $R_{y_0 +2} \dts R_{y_{t-1}+1}$, contains even number of units.}\label{subsection y0+2-yt-1+1} \ 

 Consider a row  $R\in \{R_{y_0 +2} \dts R_{y_{t-1}+1}\}$. According to (\ref{finalmatrix}),  there exists $i \in \{1 \dts t-2\},  y_i,y_{i+1}, h$, such that $y_i < h \leq y_{i+1}$ and $R = (a_k(y_{i+1}-h,y_t) \dts a_k(y_t-h,y_t))$ (where we excluded the zero's left of the term $a_k(y_{i+1}-h,y_t)$). Furthermore, by Lemma \ref{2-base of y_1,...,y_t} there exist sequences $b^i_1<\cdots<  b^i_{k_i}$, $b^{i+1}_1 <\cdots<  b^{i+1}_{k_{i+1}}$, such that $b^i_j,b^{i+1}_j \in \{ \la_1 \dts \la_{\nu}\}$, and $y_i = i_0 + 2^{b^i_{k_i}} +\cdots + 2^{b^i_1}, \ y_{i+1} = i_0 +2^{b^{i+1}_{k_{i+1}}} +\cdots + 2^{b^{i+1}_1}$. Moreover, $2^{b^i_{k_i}} +\cdots + 2^{b^i_1} < h-i_0 \leq 2^{b^{i+1}_{k_{i+1}}} +\cdots + 2^{b^{i+1}_1}$ and set $x:= h-i_0 > 0$.  Lemma \ref{2-base of y_1,...,y_t} implies that, $w\in R$ if and only if there exists sequence $a_1 <\cdots<  a_k$, such that $a_i \in \{\la_1 \dts \la_{\nu}\}, \ 2^{a_k} +\cdots +2^{a_1} \geq 2^{b^{i+1}_{k_{i+1}}} +\cdots + 2^{b^{i+1}_1}$ and $w = a_k(2^{a_k} +\cdots + 2^{a_1} - x,y_t)$. 

\begin{remark}
Contrary to the rows $R_1 \dts R_{y_0+1}$, where the arbitrary row element took the form $ a_k((i_0 - h) +2^{a_k} +\cdots +2^{a_1} ,0,y_t)$, where $i_0 - h > 0$, in the rows  $R_{y_0+2} \dts R_{y_{t-1} +1}$ we encounter a scenario where   the arbitrary row element is of the form $ a_k(2^{a_k} +\cdots +2^{a_1}-(h-i_0),0,y_t)$, where $h-i_0 >0$. As a result, Lemma \ref{Elementary1} does not apply in the same way as in rows $R_1 \dts R_{y_0+1}$, requiring necessary adjustments.
\end{remark}

  From Lemma \ref{2-base of y_1,...,y_t} we deduce that, for every $j \in \{i+1 \dts t\}$, there exist a sequence $d_1 <\cdots<  d_r$, such that $d_s\in \{\la_1 \dts \la_{\nu}\}$ and $y_t -h = y_j -h + 2^{d_r} +\cdots + 2^{d_1}$. As a consequence, the following are equivalent 
\begin{enumerate}

    \item 
$y_t-h$ has a 2-digit $\leq 2^{\la_1-1}$.

    \item 
$y_j - h$ has a 2-digit $\leq 2^{\la_1-1}$, for all $j \in \{i+1 \dts t\}$.

    \item
There exists $j_0 \in \{i+1 \dts t\}$ such that $y_{j_0} - h$ has a 2-digit $\leq 2^{\la_1-1}$.
\end{enumerate}

  From this point forward, we may assume that $y_t-h$ and  thus  $y_{i+1}-h \dts y_t-h$ do  not contain a 2-digit $\leq 2^{\la_1-1}$. This assumption arises because, in the contrary case,  Lemma \ref{formula} implies $R =  (0 \dts 0)$, which have zero-even number of units. \

  Summarizing the preceding discussion combined with Lemma \ref{formula}, we arrive at the following corollary

\begin{corollary}\label{B_R}
The number of units in row R is equal to following cardinality 
{\small\begin{align*}
|&B_R:= \{2^{b_s} +\cdots + 2^{b_1} : s\geq 1, \ b_1 <\cdots<  b_s, \ b_i \in \{\la_1 \dts \la_{\nu}\}, \ 2^{b_s} +\cdots + 2^{b_1} \geq \\
&2^{b^{i+1}_{k_{i+1}}} +\cdots + 2^{b^{i+1}_1} \textit{and} \ 2^{b_s} +\cdots + 2^{b_1} - x \  \textit{does not have  2-digit}\in \{2^{\la_2} \dts 2^{\la_{\nu}}\} \}|. \    
\end{align*}}
\end{corollary}
    
 To proceed now with the study of cardinality $|B_R|$, it is necessary to distinguish between the cases where  $2^{b^i_{k_i}} +\cdots + 2^{b^i_1} < x <2^{b^{i+1}_{k_{i+1}}} +\cdots + 2^{b^{i+1}_1}$ or $x = 2^{b^{i+1}_{k_{i+1}}} +\cdots + 2^{b^{i+1}_1}$. Starting with the first case \ 

\begin{enumerate}
    \item 
$x = 2^{c_k} +\cdots + 2^{c_1}$, where $c_1 <\cdots<  c_k \leq b^{i+1}_{k_{i+1}}$.
    \item 
Consider $M \in \{1 \dts k\}$, such that $c_M \notin \{\la_1 \dts \la_{\nu}\}$ and $M$ = maximum with this property.
\end{enumerate}
Notice that there exists $j \in \{1 \dts k\}$ such that $c_{j}\notin \{\la_1 \dts \la_{\nu}\}$; otherwise, we could find $a \in \{1 \dts t\}$ such that $y_i < y_a = x < y_{i+1}$, leading  to a contradiction. \ 
For subsequent use, consider the following element
\begin{equation}
y_M= \begin{cases}
2^{c_1}, \ &  M = 1 \\
\sm_{\substack{c_1 \leq \la \leq c_M \\ \la \neq c_2 \dts c_M}}2^{\la} = \widehat{2^{c_M}}+\cdots +\widehat{2^{c_2}} +\cdots + 2^{c_1}, \ & M \geq 2 
\end{cases}.   
\end{equation}

\begin{lemma}\label{2-digit,epsilon,delta}
For $j\in \{1,2\}$ set $\Lambda_j = \{2^{\la_i}: i \in\{j \dts \nu\}, \  \la_i <c_M\}$. Then the following holds

\begin{enumerate}

    \item 
Suppose that $y_M$ has a 2-digit in $\Lambda_2$, then there exists a unique sequence $\delta_1 <\cdots<  \delta_m$, such that $\delta_i \in \Lambda_2$ and $y_M + 2^{\delta_m} +\cdots + 2^{\delta_1}$ do not have a 2-digit in $\Lambda_2$. The same result applies for $\Lambda_1$ instead, and in this instance, denote the sequence as $\epsilon_1 <\cdots<  \epsilon_n$.

    \item 
$\max{\{y_M + 2^{\delta_m} +\cdots + 2^{\delta_1}, y_M + 2^{\epsilon_n} +\cdots + 2^{\epsilon_1}\}} < 2^{C_M + 1}$.

    \item
Suppose that $y_M$ has  $2^{\la_1}$ as a 2-digit, and another 2-digit in $\Lambda_2$. Then $2^{\delta_m} +\cdots + 2^{\delta_1} \neq 2^{\epsilon_n} +\cdots + 2^{\epsilon_1}.$ 
\end{enumerate}
\end{lemma}

\begin{proof}
(1) is a consequence of Lemma \ref{Elementary1}, (2) stems from Remark \ref{Remark-inequality} (1), and (3) becomes evident upon observing that $\epsilon_1 = \la_1$. 
\end{proof}

\begin{lemma}\label{x<b(i+1)}
If  $2^{b^i_{k_i}} +\cdots + 2^{b^i_1} < x < 2^{b^{i+1}_{k_{i+1}}} +\cdots + 2^{b^{i+1}_1}$, then $R$ has an even number of units.
\end{lemma}

\begin{proof}
For technical reasons, there are twelve cases we need to distinguish. Some of these cases lead to contradictions, while the rest are treated similarly. Below, we  display the case tree\ 

($M>1$ or $M =1$) $\longrightarrow$ ($c_1=\la_1$ or $c_1 > \la_1$ or $c_1 < \la_1$) $\longrightarrow$
($y_M$ have or doesn't have 2-digit in $\Lambda_2$). \\

 \fbox{ \textbf{Case.} $M>1, \ c_1 = \la_1$ and $y_M$ have a 2-digit in $\Lambda_2$}

We may assume $B_R\neq \emptyset$, otherwise it is clear. Pick an element $2^{b_s} +\cdots + 2^{b_1} \in B_R$ and  consider $s_0 \in \{1 \dts s\}$ such that $b_{s_0} > c_M$ and $s_0$ is the minimum with such a property. Note that such $s_0$ exists; otherwise $b_i < c_M$ for all $i= 1 \dts s$ (the strict inequality comes from the fact that $c_M \notin \{\la_1 \dts \la_{\nu}\}$) and thus $2^{b_s} +\cdots + 2^{b_1} < 2^{c_M} \leq x < 2^{b^{i+1}_{k_{i+1}}} +\cdots + 2^{b^{i+1}_1}$, leading to contradiction.  Now we may compute
\begin{align*}
&2^{b_s} +\cdots + 2^{b_1} - x = \\
&=2^{b_s} +\cdots + 2^{b_{s_0+1}} + (2^{b_{s_0}} - 2^{c_1}) - (2^{c_k} +\cdots + 2^{c_2}) + (2^{b_{s_0-1}}+\cdots + 2^{b_1})   \\
&=\sm_{\la = s_0 +1}^s2^{b_{\la}} - \sm_{\substack{M+1\leq i \leq k:\\ c_i\geq b_{s_0}}}2^{c_i} \hspace{0.4cm}  +  \hspace{0.4cm} \sm_{\mathclap{\substack{c_M < \la \leq b_{s_0}-1: \\ \la \notin \{c_i: M+1 \leq i \leq k, \  c_i < b_{s_0}\}}}}2^{\la} \hspace{0.3cm} + \hspace{0.3cm} y_M +  \sm_{\la = 1}^{s_0-1}2^{b_{\la}}.
\end{align*}
Note that if the conditions for an index within a sum are empty, then we mean that the corresponding sum is zero. Let us denote the third sum by $\Tilde{y}_M$; $\Tilde{y}_M+y_M = 2^{b_{s_0}-1} +\cdots + \sm_{M+1 \leq i \leq k}\widehat{2^{c_i}}+\cdots + \widehat{2^{c_M}}+\cdots +\widehat{2^{c_2}} +\cdots + 2^{c_1}$. Below we present some further observations

\begin{enumerate}

    \item 
$\Tilde{y}_M + y_M + \sm_{\la = 1}^{s_0-1}2^{b_{\la}} < 2^{b_{s_0}}$.

    \item 
$b_{s_0} \notin \{c_1 \dts c_k\}$ and $\sm_{\la = s_0 +1}^s2^{b_{\la}} = \sm_{\substack{M+1\leq i \leq k:\\ c_i\geq b_{s_0}}}2^{c_i} = \sm_{\substack{M+1\leq i \leq k: \\ c_i> b_{s_0}}}2^{c_i}$.

    \item
The element $\Tilde{y}_M$ does not contain a 2-digit in the set $\{2^{\la_2} \dts 2^{\la_{\nu}}\}$.  

    \item
$s_0 > 1$ and $\sm_{\la = 1}^{s_0 -1}2^{b_{\la}} \in \{2^{\delta_m} +\cdots + 2^{\delta_1}, 2^{\epsilon_n} +\cdots + 2^{\epsilon_1}\}$.    

\end{enumerate}

\textit{Proof of} $(1)$. From definition of $s_0$ and the fact that $c_M \notin \{\la_1 \dts \la_{\nu}\}$, we get that $b_1 <\cdots<  b_{s_0-1} < c_M$. Consequently it is easy to observe  that $y_M + \sm_{\la = 1}^{s_0-1}2^{b_{\la}} < 2^{c_M + 1} \leq 2^{b_{s_0}}$. Moving forward, we consider two cases. In case where $b_{s_0}-1 = c_M$, then $\Tilde{y}_M = 0$ and the result follows. In case $b_{s_0}-1 > c_M$, then $\Tilde{y}_M + y_M + \sm_{\la = 1}^{s_0-1}2^{b_{\la}} \leq \sm_{\la = c_M +1}^{b_{s_0}-1}2^{\la} + y_M  + \sm_{\la = 1}^{s_0-1}2^{b_{\la}} < \sm_{\la = c_M +1}^{b_{s_0}-1}2^{\la} + 2^{c_M+1} = 2^{b_{s_0}}$. \\ 

\textit{Proof of} $(2)$. For convenience denote the sums by $S_1,S_2,S_3$  respectively. Let's begin by making the following observation, $\{\la : s_0 + 1 \leq \la \leq s\} = \emptyset$ \ if and only if $\{i: M+1 \leq i \leq k, \  c_i \geq b_{s_0}\} = \emptyset$. Indeed, suppose the first set is empty and the second is non-empty, then 
\begin{align*}
0 < 2^{b_s} +\cdots + 2^{b_1} - x &=  - \sm_{\substack{M+1\leq i \leq k:\\ c_i\geq b_{s_0}}}2^{c_i} + \Tilde{y}_M + y_M + \sm_{\la = 1}^{s_0 - 1}2^{b_{\la}} \\
&< - \sm_{\substack{M+1\leq i \leq k:\\ c_i\geq b_{s_0}}}2^{c_i} +  2^{b_{s_0}} \leq 0.
\end{align*}
this leads to contradiction. Now suppose the second set is empty and the first is non-empty. Then $2^{b_s} +\cdots + 2^{b_1} - x = \sm_{\la = s_0 +1}^s2^{b_{\la}} + \Tilde{y}_M + y_M + \sm_{\la = 1}^{s_0 - 1}2^{b_{\la}}$. Since the first set is non-empty, (1) compels the element  $2^{b_s} +\cdots + 2^{b_1} - x$ to have  2-digits in the set $\{\la_2 \dts \la_{\nu}\}$, contradicting  the assumption that $2^{b_s} +\cdots + 2^{b_1} \in B_R$. Consequently, $S_1 = 0  \ \Longleftrightarrow \ S_2 = 0$. Lets assume now that both of the sets are non-empty and thus $S_1,S_2 \neq 0$. Suppose for a contradiction that $S_1 \neq S_2$\ 

Case $S_1 > S_2$.  Then by Lemma \ref{2-digits-1}, $S_1 - S_2 = \sm_{\la \geq b_{s_0}}s_{\la}2^{\la}$ (base 2 representation) and $S_1-S_2$ has a 2-digit in the set $\{2^{b_{s_0+1}} \dts 2^{b_s}\}\cup\{2^{c_i}: M+1 \leq i \leq k, \ c_i \geq b_{s_0}\} \subseteq \{2^{\la_2} \dts 2^{\la_{\nu}}\}$ (the inclusion $"\subseteq"$ follows from the definition of $c_M$ and the fact that $c_1 \geq \la_1$). Combining the earlier statements with property (1) implies that  the element $2^{b_s} +\cdots + 2^{b_1} - x = S_1 - S_2 + \Tilde{y}_M + y_M + \sm_{\la = 1}^{s_0 - 1}2^{b_{\la}}$ has a 2-digit in the set $\{2^{\la_2} \dts 2^{\la_{\nu}}\}$,  contradicting the assumption that $2^{b_s} +\cdots + 2^{b_1} \in B_R$. \

Case $S_1 < S_2$. Then by Lemma \ref{2-digits-1}, $S_2-S_1$ has a 2-digit $2^x \geq 2^{b_{s_0}}$. Hence, $0 < 2^{b_s} +\cdots + 2^{b_1} - x = -(S_2 - S_1) + \Tilde{y}_M + y_M + \sm_{\la = 1}^{s_0 - 1}2^{b_{\la}} < -2^x + 2^{b_{s_0}} \leq 0$, leading to contradiction.\ 

 As a result $S_1 = S_2$. It remains to prove that $b_{s_0} \notin \{c_1 \dts c_k\}$. We recall that $b_{s_0} > c_M > \cdots > c_1$. In case $\{i: M+1 \leq i \leq k, \  c_i \geq b_{s_0}\} = \emptyset$, then it is straightforward. In different case, utilizing the uniqueness of the base 2 representation in the equality $\sm_{\la = s_0 +1}^s2^{b_{\la}} = \sm_{\substack{M+1\leq i \leq k:\\ c_i\geq b_{s_0}}}2^{c_i}$, we conclude that $b_{s_0} \notin \{c_{M+1} \dts c_k\}$.\ 

 \textit{Proof of} (3). Applying (2), we obtain  $2^{b_s} +\cdots + 2^{b_1} - x =  \Tilde{y}_M + y_M + \sm_{\la = 1}^{s_0 - 1}2^{b_{\la}}$, where $y_M + \sm_{\la = 1}^{s_0-1}2^{b_{\la}} < 2^{c_M + 1}$, and $2^{b_s} +\cdots + 2^{b_1} \in B_R$. Thus, $\Tilde{y}_M$ must not contain a 2-digit in the set $\{2^{\la_2} \dts 2^{\la_{\nu}}\}$.\\

 \textit{Proof of} (4). Assume $s_0 =1$, then  $\sm_{\la = 1}^{s_0 -1}2^{b_{\la}} = 0$ and $2^{b_s} +\cdots + 2^{b_1} - x = \Tilde{y}_M + y_M$. This implies that the element $b:= 2^{b_s} +\cdots + 2^{b_1} - x$ contains 2-digit in $\Lambda_2$ (see case assumption for $y_M$), which leads to contradiction. Consequently,  $b = \Tilde{y}_M + y_M + \sm_{\la = 1}^{s_0 -1}2^{b_{\la}}$, where $y_M + \sm_{\la = 1}^{s_0 -1}2^{b_{\la}} < 2^{C_M +1}$ and $b$ does not have a 2-digit in $\Lambda_2$. These conditions combined with $(3)$, forces $y_M + \sm_{\la = 1}^{s_0 -1}2^{b_{\la}}$ to also not have  a 2-digit in $\Lambda_2$. According to this  we distinguish the following two cases \ 

In case $\la_1 \notin \{b_1 \dts b_{s_0-1}\}$. Then $\{b_1 \dts b_{s_0-1}\} \subseteq \Lambda_2$ and Lemma \ref{2-digit,epsilon,delta} implies $\sm_{\la = 1}^{s_0 -1}2^{b_{\la}} = 2^{\delta_m} +\cdots + 2^{\delta_1}$. \ 

In case $\la_1 \in \{b_1 \dts b_{s_0-1}\}$. Then $\la_1 = b_1$ and by case assumption $c_1  = \la_1$ we get that  $y_M + \sm_{\la = 1}^{s_0 -1}2^{b_{\la}}$ does not have  a 2-digit in $\Lambda_1$. Consequently,   Lemma \ref{2-digit,epsilon,delta} implies  $\sm_{\la = 1}^{s_0 -1}2^{b_{\la}} = 2^{\epsilon_n} +\cdots + 2^{\epsilon_1}$. \\

 Now, we proceed by defining the following elements and set 
\begin{enumerate}
    \item[$a)$] 
$\Tilde{y}_{\la_d,M}:=\hspace{0.4cm}\sm_{\mathclap{\substack{c_M < \la \leq \la_d-1, \\ \la \notin \{c_i : M+1 \leq i \leq k, \ c_i < \la_d\}}}}2^{\la}$,  \ $z_{\la_d,M} = \sm_{\substack{M+1 \leq i \leq k: \\ c_i > \la_d}}2^{c_i}$, $d\in \{2 \dts \nu\}$.
    \item[$b)$]
$D:= \{d \in\{2 \dts \nu\} : \  \la_d > c_M, \ \la_d \notin \{c_1 \dts c_k\} \  \textit{and} \ \Tilde{y}_{\la_d,M} \ \textit{does not }$ $\textit{have a 2-digit}\in \{2^{\la_2} \dts 2^{\la_{\nu}}\} \}.$ 
\end{enumerate}
 $D$ is non-empty, since for every $2^{b_s} +\cdots + 2^{b_1} \in B_R$, by $(2), (3)$ we get that $b_{s_0} \in D$. \    

 \textbf{Claim}. $B_R = \bigcup\limits_{d \in D} \{z_{\la_d,M} + 2^{\la_d}  + 2^{\delta_m} +\cdots + 2^{\delta_1}, \ z_{\la_d,M}  + 2^{\la_d} + 2^{\epsilon_n} +\cdots + 2^{\epsilon_1} \}$.    
 
Note that the first sum is zero if the index conditions  are empty. Specifically we  have a disjoint union  of sets with two distinct elements (see Lemma \ref{2-digit,epsilon,delta}), thus $|B_R| = 2|D|$, which proves the lemma. \\  

 \textit{Proof of Claim}.  The inclusion $"\subseteq"$ is immediate from $(2), (3), (4)$. Conversely, consider $d\in D$ to prove $S:= \sm_{\substack{M+1 \leq i \leq k: \\ c_i > \la_d}}2^{c_i} + 2^{\la_d} + 2^{\delta_m} +\cdots + 2^{\delta_1} \in B_R$ (similarly the other element). It is easy to see that $\delta_1 <\cdots<  \delta_{m} < \la_d$ and $\{\delta_1 \dts \delta_m, \la_d, c_{M+1} \dts c_k\} \subseteq \{\la_1 \dts \la_{\nu}\}$. It remains  to prove   that $S \geq 2^{b^{i+1}_{k_{i+1}}} +\cdots + 2^{b^{i+1}_1}$ and the element $S - x$  does not have  2-digit in $\{2^{\la_2} \dts 2^{\la_{\nu}}\}$. Since $d\in D$, we get that:
\begin{align*}
x = 2^{c_k} +\cdots + 2^{c_1} &= \sm_{\substack{M+1 \leq i \leq k: \\ c_i > \la_d}}2^{c_i} +  \sm_{\substack{M+1 \leq i \leq k: \\ c_i < \la_d}}2^{c_i} + 2^{c_M} +\cdots + 2^{c_1} \\
&< \sm_{\substack{M+1 \leq i \leq k: \\ c_i > \la_d}}2^{c_i} + 2^{\la_d} \leq S.
\end{align*}
Hence, $h-i_0 = x < S$  and $y_i = i_0 + 2^{b^i_{k_i}} +\cdots + 2^{b^i_1} < h < i_0 + S \in \{y_1 \dts y_{t}\}$ (see Lemma \ref{2-base of y_1,...,y_t}). Thus, $i_0 + S \geq y_{i+1}$ and so $S \geq 2^{b^{i+1}_{k_{i+1}}} +\cdots + 2^{b^{i+1}_1}$. Furthermore, similar  to the computations  at the start (for $\la_d$ instead of $b_{s_0}$) we have
\begin{align*}
S-x &=  \sm_{\substack{M+1 \leq i \leq k: \\ c_i > \la_d}}2^{c_i} - \sm_{\substack{M+1 \leq i \leq k: \\ c_i \geq \la_d}}2^{c_i}  + \Tilde{y}_{\la_d,M} + y_M + 2^{\delta_m} +\cdots + 2^{\delta_1} \\
&=  \Tilde{y}_{\la_d,M} + y_M + 2^{\delta_m} +\cdots + 2^{\delta_1},
\end{align*}
where from (1), (2) of Lemma \ref{2-digit,epsilon,delta} and the fact that $d\in D$, the element $S-x$ must not have a 2-digit in the set $\{2^{\la_2} \dts 2^{\la_{\nu}}\}$.  \ 

The  case  $M>1, \ c_1 = \la_1$ and $y_M$ have a 2-digit in $\Lambda_2$, has been fully proved.  \

In the remaining cases where $M\geq 1, c_1 \geq \la_1$ it is easy to see that the properties $(1), (2), (3)$ still hold. However, the satisfaction of property (4) and, consequently, the representation of $B_R$ as a disjoint union, varies across cases (with the set $D$ remaining unchanged).  Let's specifically discuss the alterations in property $(4)$. For convenience set $S_0 := \sm_{\la = 1}^{s_0 -1}2^{b_{\la}}$.
\begin{enumerate}
    
    \item 
If $M>1, c_1 = \la_1$ and $y_M$ does not have a 2-digit in $\Lambda_2$. Then $s_0 \geq 1$ and $S_0 \in \{2^{\epsilon_n} +\cdots + 2^{\epsilon_1}, 0\}$.

    \item 
If $M>1, c_1 > \la_1$ and $y_M$ have a 2-digit in $\Lambda_2$. Then $s_0 >2$ and  $S_0 \in \{2^{\delta_m} +\cdots + 2^{\delta_1}, 2^{\delta_m} +\cdots + 2^{\delta_1} + 2^{\la_1}\}$.  

    \item 
If $M\geq1$, $c_1 > \la_1$ and $y_M$ does not have a 2-digit in $\Lambda_2$. Then $s_0 \in \{1,2\}$ and $S_0 \in \{0,2^{\la_1}\}$. \\

\end{enumerate}

 \fbox{ \textbf{Case.} $M>1, \ c_1 < \la_1$ and $y_M$ arbitrary} \ 

We claim that $B_R  = \emptyset$. Indeed assuming that $B_R \neq \emptyset$, pick an element $2^{b_s} +\cdots + 2^{b_1} \in B_R$. Then $x = 2^{c_k} +\cdots + 2^{c_1} < 2^{b_s} +\cdots + 2^{b_1}$, where $c_1 < \la_1 \leq b_1$. On the other hand, Lemma \ref{2-digits-1} (3), forces the element $2^{b_s} +\cdots + 2^{b_1} - x = (2^{b_s} +\cdots + 2^{b_1} + i_0) - h$, to have $2^{c_1} (\leq 2^{\la_1-1})$ as a 2-digit. This come in contrary with the convention we made above Corollary \ref{B_R}. \\ 

 \fbox{ \textbf{Case.} $M=1, \ c_1 \leq \la_1$ and $y_M$ arbitrary} \ 

The case $M = 1$ and $c_1 = \la_1$, contradicts the condition  $c_M \notin \{\la_1 \dts \la_n\}$. Regarding the case $M= 1, c_1 < \la_1$,  we can demonstrate that $B_R = \emptyset$ similar to the case $M> 1, c_1 < \la_1$.\\

 \fbox{ \textbf{Case.} $M=1, \ c_1 > \la_1$ and $y_M$ have a 2-digit in $\Lambda_2$} \ 

Then $y_M = 2^{c_1}$ and definition of $M$ implies $c_1 = c_M \notin \{\la_1 \dts \la_{\nu}\}$, which comes in contrary with the case assumption.
\end{proof}

\begin{lemma}\label{Elementary2}
Consider  two sequences  $\la_1 <\cdots<  \la_{\nu}$ $(\nu \geq 2)$, $\la_1 \leq a_1 <\cdots< a_n$, such that $a_i \in \{\la_1 \dts \la_{\nu}\}$ and $ 2^{a_n} +\cdots + 2^{a_1} < 2^{\la_1} +\cdots + 2^{\la_{\nu}}$. Then, there is a unique sequence $b_1 <\cdots< b_m$ such that $b_i \in \{\la_1 \dts \la_{\nu}\}, \  2^{a_n} +\cdots + 2^{a_1} < 2^{b_1} +\cdots + 2^{b_m}$ and the element $(2^{b_1} +\cdots + 2^{b_m})-(2^{a_n} +\cdots + 2^{a_1})$ does not have any 2-digit inside the set $\{2^{\la_2} \dts 2^{\la_{\nu}}\}$.
\end{lemma}

For proof of above Lemma \ref{Elementary2} see appendix.

\begin{corollary}\label{x=b(i+1)}
If  $2^{b^i_{k_i}} +\cdots + 2^{b^i_1} < x = 2^{b^{i+1}_{k_{i+1}}} +\cdots + 2^{b^{i+1}_1}$, then $R$ contains exactly two  units.
\end{corollary}

\begin{proof}
According to (\ref{finalmatrix}), $R = (1,a_k(y_{i+2}-y_{i+1},y_t) \dts a_k(y_{t}-y_{i+1},y_t))$, where Lemma \ref{2-base of y_1,...,y_t} implies that  $\{y_{i+2} \dts y_t\} = \{i_0 + 2^{b_k} +\cdots + 2^{b_1} : k \geq 1, \ b_1 <\cdots<  b_k, \ b_i\in \{ \la_1 \dts \la_{\nu}\}$ and $2^{b_k} +\cdots + 2^{b_1} >  2^{b^{i+1}_{k_{i+1}}} +\cdots + 2^{b^{i+1}_1}\}$. Therefore, by Lemma \ref{Elementary2}, $R$ contains exactly two units.
\end{proof}

\begin{corollary}
Rows $R_{y_0 +2} \dts R_{y_{t-1}+1}$, contains even number of units.
\end{corollary}

\begin{proof}
It is immediate from the discussion at the start of the   subsection \ref{subsection y0+2-yt-1+1}, employing Lemma \ref{x<b(i+1)} and Corollary \ref{x=b(i+1)}.
\end{proof}

\subsubsection{Rows $R_{y_{t-1}+2} \dts R_{y_t}$, contains even number of units.}\label{Subsection Ryt-1+2-Ryt}

\begin{corollary}
Rows $R_{y_{t-1}+2} \dts R_{y_t}$, does not contain any units.    
\end{corollary}

\begin{proof}
Consider a row $R\in \{R_{y_{t-1}+2} \dts R_{y_t}\}$. Based on (\ref{finalmatrix}) there exists $h$ such that $y_{t-1}+1 < h < y_t$ and $R = (a_k(y_t-h,y_t))$ (where we excluded the zero's left of the term $a_k(y_t-h,y_t)$). Furthermore, from Lemma \ref{2-base of y_1,...,y_t} we get that $y_{t-1} = i_0 + 2^{\la_{\nu}} +\cdots + 2^{\la_2}, \ y_t = i_0 + 2^{\la_{\nu}} +\cdots + 2^{\la_1}$. We may express $h$ in the form $h = y_{t-1} + x$ for some non negative integer $x$, where  $y_{t-1}+1 < h < y_t $, moreover $1<x < 2^{\la_1}$. Therefore, $a_k(y_t-h,y_t) = a_k(2^{\la_1}-x,y_t) = 0$ (the last equality follows from Lemma \ref{formula}). 
\end{proof}

\section{The case $\chK=2, a=$\textit{even}, $b=$\textit{odd}}\label{Section,a=even,b=odd}
This case is similar to the one where $b$ was an even integer, and thus we will highlight only the modifications we need to make.

\begin{lemma}\label{odd-Equations123} Consider $a$ even, $b$ odd and set $\gamma = a-b$. Then the following are equivalent
\begin{enumerate}
    \item 
Coefficients $ \{c_{b_2 \dts b_{d+2}}: (b_2 \dts b_{d+2})  \in B(\la,\mu)\} \subseteq K$ satisfy equations (\ref{Equations1}), (\ref{Equations2}), (\ref{Equations3}).
    \item

\begin{enumerate}

    \item 
For  $s \in \{b-\min\{b,d\} \dts  b\}$, the set $\{c_{b_2 \dts b_{d+2}}: (b_2 \dts b_{d+2}) \in \Blm_s\} $ is singleton, and we denote its unique element by $c_{b-s}$. Particularly,  $\{c_{b_2 \dts b_{d+2}}: (b_2 \dts b_{d+2}) \in \Blm\} = \{c_{b-s}: s \in \{b-\min\{b,d\} \dts  b\}\}$.

    \item 
$c_{q, 0,b_4 \dts b_{d+2}} = c_{q-1,1,b_4 \dts b_{d+2}}$, for all $q, b_4 \dts b_{d+2}$ such that, $b-\min\{b,d-1\} \leq q \leq b$, $q = $odd, $0 \leq b_s \leq 1 \ \forall s \geq 4$  and $\sm_{i  = 4}^{d+2}b_i  = b -q$.

    \item
For all $t, r$  such that $1 \leq t \leq \min\{b,d\}, \  t \leq r \leq \min\{b,d\}$, we have
\begin{equation}\label{odd-Equations1-3}
\sm_{0 \leq j \leq t }\tbinom{\gamma + r + t - j }{t-j}\tbinom{r}{j}c_{r-j} = 0 . \\     
\end{equation}
\end{enumerate}    
\end{enumerate}
\end{lemma}

We may observe that equations $2)b)$ from  Lemma \ref{odd-Equations123} are equivalent to equations $c_s = c_{s+1}$ for all $s$ such that, $0 \leq s \leq \min\{b,d-1\}$ and $s=$even. Moreover, since $\gamma = a-b, b$  is odd,  the last equations (equivalently  equations $2)b)$) are special case of equations (\ref{odd-Equations1-3})  for $r = s+1$ and $t=1$ (where $s+1 \leq b$, because $s=even$ and $b=odd$). \

Moving forward, set $l' = \min\{b,d\}$ and  define the linear system of equations $\Tilde{E}_k := \{\Tilde{E}_{i,k}\}_{i = 1}^{l'}$, where  $\Tilde{E}_{i,k} := \left\{ \sm_{j = 0}^sa_{\gamma}(s-j,i)\epsilon(j,i)c_{i-j} = 0 \right\}_{s = 1 }^i$  and $a_{\gamma}(x,y) := \tbinom{\gamma + y + x}{x}, \ \epsilon(x,y):= \tbinom{y}{x}$. Also set $j_0$ to be the maximum integer in the interval $[0,l']$ that the $\gamma = a-b$ constitutes a 2-complement of it.  Similar to the case where $b$ was an even integer we have the following proposition

\begin{proposition}\label{odd-CharacterSystem}
Consider $a$ even and $b$ odd. Let $\Phi \in \Hom_S(D(\la),\D(\mu))$, then $\Phi \in \im\pi_{\D(\la)}^{*} (\simeq \Hom_S(\D(\la),\D(\mu)))$ if and only if there exists coefficients $c_{0} \dts c_{l'} \in K$ such that 
\begin{enumerate}
    \item 
$\Phi = \sm_{i = 0}^{l'}c_{i}F_{i}$. 
    \item
The coefficients $c_{0} \dts c_{l'} \in K$ satisfies the equations $\Tilde{E}_k = \{\Tilde{E}_{i,k}\}_{i = 1}^{l'}$.
\end{enumerate}
Moreover, system $\Tilde{E}_k$ is equivalent to $c_{j_0} = \cdots = c_{j_{m'}}$ and $c_i = 0$, for all $i \in \{0 \dts l'\}\setminus\{j_0 \dts j_{m'}\}$ where $\{j_0 < \dots < j_{m'}\} = \Hlmm = \{h \in \mathbb{N} :h \leq \min\{b,d\}$ and $h$ contains the base 2 representation of $j_0$\}.
\end{proposition}

\begin{proof}
The way we treat the system $\Tilde{E}_{k}$ it is exactly the same with the system $E_k$, with $\gamma$ instead of $k$.   
\end{proof}

\section{Some remarks.}

Not all homomorphisms in the space $\Hom_S(\D(a,b,1^d), \D(a+d,b))$ that are non-zero are necessarily compositions of Carter-Payne homomorphisms. We will provide a family of examples, but in order  to proceed with this demonstration, it is necessary to recall the following result

\begin{theorem}(see \cite[Theorem 3.1]{maliakas2021homomorphisms})\label{MDhooks}
Let k be an infinite field of characteristic $p>0$ and $\la, h \in \Lnr$, such that $\la = (\la_1 \dts \la_m), \ \la_m \neq 0$, $h = (a,1^d)$ and $m \geq d+2$. Additionally, set $q = \max\{i : \la_i \geq 2\}$. Assuming that $q\leq d$ and $p$ does not divide the integer $\la_q +d + 2 -q$, then $\Hom_S(\D(\la),\D(h)) = 0$. \ 
\end{theorem}

 Note that, in the case  where  $\la = (1 \dts 1)$ define $\la_q = 0$ and $ q =0$.

\begin{remark}
In its original statement,  Theorem \ref{MDhooks} provides a complete characterization of the space $\Hom_S(\D(\la),\D(h))$. However we only require the special case. 
\end{remark}

\begin{example}\label{NonCarter-Payne}
Consider $\la = (a,2,1^{d})\in \Lnr$, where $d\geq 3$ is an odd integer. Then, none of the non-zero elements of    $\Hom_S(\D(a,2,1^d), \D(a+d,2))$ is composition of Carter-Payne homomorphisms. \end{example}

\begin{proof}
We may assume from Theorem \ref{Mainresult} that $a$ is even. Let $0 \neq f \in \Hom_S(\D(\la),\D(\mu))$ and suppose there exist Carter-Payne homomorphisms say $\phi_1 \dts \phi_k$ such that $f = \phi_k \circ \cdots \circ \phi_1$. We need to distinguish some cases  based on the specific type of Carter-Payne homomorphism that  the map $\phi_1$ might represent \\

 \textbf{Case 1}. $\phi_1 : \D(a,2,1^d) \longrightarrow\D(a+1,2,1^{d-1})$. Then $k\geq 2, \ a+1$ is an odd integer and  by Theorem \ref{Mainresult} the map $\phi_n \circ \cdots \circ \phi_2: \D(a+1,2,1^{d-1}) \longrightarrow \D(a+d,2)$, is zero. Consequently $f =0$. \

 \textbf{Case 2}. $\phi_1 : \D(a,2,1^d) \longrightarrow\D(a,3,1^{d-1})$, where $3 \leq a$. Now, from the row removal theorem (see \cite{kulkarni2006ext} Proposition 1.2), we obtain 
\begin{align*}
\Hom_S(\D(a,2,1^d), \D(a,3,1^{d-1}))\simeq \Hom_{S'}(\D(2,1^d), \D(3,1^{d-1})) = 0,    
\end{align*}
where the equality follows by Theorem~\ref{MDhooks}. Consequently $f =0$.\

 \textbf{Case 3}. $\phi_1 : \D(a,2,1^d) \longrightarrow\D(a+1,1^{d+1})$.  Then $k\geq 2$ and there are two cases for the map $\phi_2$. It can be either of the form $\phi_2:\D(a+1,1^{d+1})\longrightarrow \D(a+2,1^d)$  or  $\phi_2:\D(a+1,1^{d+1}) \longrightarrow \D(a+1,2,1^{d-1})$. By Theorem \ref{MDhooks} the first form is zero, and as for the second form, it necessary that $k\geq 3$ and  by Theorem \ref{Mainresult} we get that the map $\phi_k \circ \cdots \circ \phi_3 : \D(a+1,2,1^{d-1}) \longrightarrow \D(a+d,2)$ is zero. In either case we deduce that $f = 0$.\ 

 \textbf{Case 4}. $\phi_1 : \D(a,2,1^d) \longrightarrow \D(a,2,2,1^{d-2})$. From the row removal theorem again we obtain 
\begin{align*}
\Hom_S(\D(a,2,1^d), \D(a,2,2,1^{d-2})) \simeq \Hom_{S'}(\D(1^d), \D(2,1^{d-2})) = 0,    
\end{align*}
where the equality follows from Theorem~\ref{MDhooks}.  Consequently $f =0$. \

 In each of the aforementioned cases, the conclusion is that $f = 0$. This contradicts the initial assumption.
\end{proof}

Our next remark concerns the map $\psi:= \sm_{T \in \st} \phi_T : D(\la) \longrightarrow \D(\mu)$, where $\mu = (\mu_1,\mu_2)$ and $\mu_2 \leq \la_1$. An alternate version of $\psi$ regarding Specht modules instead of Weyl modules can be found in the work of Lyle (see \cite{lyle2013homomorphisms}). The map $\psi$ has also been studied in  \cite{maliakas2023homomorphisms}, as shown in the Theorem \ref{M-D2023}. In Corollary \ref{M-D2023-C} we provide an equivalent version of Theorem \ref{M-D2023} in our concept.

\begin{theorem}\label{M-D2023}(see \cite[Theorem 3.1]{maliakas2023homomorphisms})
Let $K$ be an infinite field of characteristic $p>0$ and let $n\geq r$ be positive integers. Let $\la,\mu\in \Lambda^{+}(n,r)$ be partitions such that $\la  = (\la_1 \dts \la_m)$ and $\mu = (\mu_1,\mu_2)$, where $\la_m\neq 0$, $m\geq 2$ and $\mu_2 \leq \la_1 \leq \mu_1$. Set $\psi =\sm_{T\in \st}\phi_T$, then the following are equivalent,
\begin{enumerate}
    
    \item 
$\im\bx_{\la} \subseteq \ker\psi$. (Hence the map $\psi : D(\la)\longrightarrow \D(\mu)$ induces a non zero map $\D(\la) \longrightarrow \D(\mu)$).

    \item
$p^{l_p(\min\{\la_2,\mu_1-\la_1\})} \mid \la_1-\mu_2+1$ and $p^{l_p(\la_{i+1})} \mid \la_i+1$, for all $i\in \{2 \dts m-1\}$.

\end{enumerate}
Recall that $l_p(y)$ is the least   integer $i$ such that $p^i > y$.

\end{theorem}

\begin{remark}
In its original statement, Theorem \ref{M-D2023}   mentions only that $(2)$ implies $(1)$ (see also \cite[Lemma 3.4]{maliakas2023homomorphisms}). However the proof provided in \cite{maliakas2023homomorphisms} also covers the implication $(1) \Longrightarrow (2)$.
\end{remark}

\begin{corollary}\label{M-D2023-C}
Let $K$ be an infinite field of characteristic $2$ and let $n\geq r$ be positive integers. Consider the partitions $\la = (a,b,1^d) \in \Lambda^{+}(n,r), \ \mu = (a+d, b) \in \Lambda^{+}(2,r)$, where $b,d \geq 2$. Set $\psi =\sm_{T\in \st}\phi_T$, then the following are equivelant
\begin{enumerate}
    
    \item 
$\im\bx_{\la} \subseteq \ker\psi$. (Hence the map $\psi : D(\la)\longrightarrow \D(\mu)$ induces a non zero map $\D(\la) \longrightarrow \D(\mu)$).

    \item
$a$ is even, $b$ is odd and $2^{l_2(\min\{b,d\})}\mid a-b+1$.

    \item
$a$ is even, $b$ is odd and $j_0= 0$ is the maximum integer in the interval $[0,\min\{b,d\}]$ that the $\gamma = a-b$ constitutes a 2-complement of it.
\end{enumerate}
\end{corollary}

\begin{proof}
According to Theorem \ref{M-D2023} combined with the definition of $\la,\mu$ we have
\begin{align*}
(1) \ \Longleftrightarrow \  2^{l_2(\min\{b,d\})}\mid a-b+1, \ 2\mid b+1  \ \Longleftrightarrow \ (2).    
\end{align*}
It remains to prove that $(3) \Longleftrightarrow (1)$. Initially $\im\bx_{\la} \subseteq \ker\psi \Longleftrightarrow \psi \in \im\pi_{\D(\la)}^{*}$ (see exact sequence (\ref{Relations})).

($(3) \Longrightarrow \psi \in \im\pi_{\D(\la)}^{*}$). From Definition \ref{Maintableaux}, $(\ref{F_h})$ and disjoint union (\ref{DisUnion1}) we have  that $\psi = \sm_{i = 0}^{l'}F_i, \ l' = \min\{b,d\}$. Set $c_0 = \dots = c_{l'} = 1$. Then  from assumption $(3)$ and Proposition \ref{odd-CharacterSystem} we have that $\psi \in \im\pi_{\D(\la)}^{*}$.
    
($\psi \in \im\pi_{\D(\la)}^{*} \Longrightarrow (3)$). Since $Im\bx_{\la} \subseteq ker\psi$, then the map $\psi: D(\la) \longrightarrow \D(\mu)$ induces a non zero map $\D(\la) \longrightarrow \D(\mu)$, and so from Theorem \ref{Mainresult}, $a$ must be even. From Definition \ref{Maintableaux} and Proposition \ref{Base} we have that the set $\{F_i: i = 0 \dts l'\}$ is linear independent $(*)$. Assume that $b$ is even, then from Proposition \ref{CharacterSystem}   there exist coefficients $c_{r_0} \dts c_{r_{l}} \in K$ (for $r_0 \dts r_l$, see Definition \ref{MainDefinition}) such that
\begin{align*}
\sm_{i =0}^{l'}F_i = \psi = \sm_{i = 0}^{l}c_{r_i}F_{r_i},  
\end{align*}
Above relations combined with $(*)$  leads to contradiction. Consequently, $b$ must be odd. Since $\sm_{i =0}^{l'}F_i = \psi \in \im\pi_{\D(\la)}^{*}$, then from $(*)$ and Proposition \ref{odd-CharacterSystem}, the coefficients $c_0 = \dots = c_{l'} = 1$ must satisfies the equations $\Tilde{E}_k$, thus $j_0 = 0$. 
\end{proof}


\section{Appendix}

\begin{proof}(\textbf{Lemma \ref{2complement}}).
(1). Follows immediately from Corollary \ref{LucasCor1} (3). \

(2). For $i =  h$ or $h = 0$ its trivial. Suppose, that $1 \leq h< i$. Now consider the base 2 representations, $i = \sm_{\la = 0}^mi_{\la}2^{\la}, \ h = \sm_{\la = 0}^nh_{\la}2^{\la}, \ i-h = \sm_{\la = 0}^{m'}a_{\la}2^{\la}$ $(a_{m'} = i_m = h_n = 1)$, where $n,m' \leq m$.
It is sufficient to show $h_{\la} \leq i_{\la}$ for all $\la = 0 \dts n$, for then by Theorem \ref{Lucas}, it is clear that $\epsilon(h,i) = 1$. Since $k$ is a 2-complement of $i,i-h$, we obtain 
\begin{equation*}
\sm_{\la \geq m+1}t_{\la}2^{\la} + \sm_{\substack{0\leq \la \leq m \\ i_{\la} = 0}}2^{\la} = k =   \sm_{\la \geq m'+1}t_{\la}'2^{\la} + \sm_{\substack{0\leq \la \leq m' \\ a_{\la} = 0}}2^{\la},    
\end{equation*}
for some $t_{\la}, t_{\la}' \in \{0,1\}$. Based on the aforementioned relations and the uniqueness of base 2 representations, we can conclude that, $a_{\la} = 0 \ \Longleftrightarrow \ i_{\la} = 0$ for all $\la \in \{0 \dts m'\}$. Using this, we may compute that $i-h = \sm_{\la = 0}^{m'}a_{\la}2^{\la} = \sm_{\substack{0\leq \la \leq m' \\ a_{\la} = 1}}2^{\la} = \sm_{\substack{0\leq \la \leq m' \\ i_{\la} = 1}}2^{\la}$, equivalently $h = \sm_{\la \geq m'+1}i_{\la}2^{\la}$ and so  $h_{\la} \leq i_{\la}$ for all $\la = 0 \dts n$.\ 

(3).   For $i= 0$ it is trivial. Let $i \geq 1$ and consider the  base 2 representations $i = \sm_{\la  = 0}^mi_{\la}2^{\la}$ $(i_m = 1)$, $i + k  = \sm_{\la = 0}^nt_{\la}2^{\la}$, where $n \geq m, \ t_{\la}\in \{0,1\}$. Since $k$ is not a 2-complement of i, there must exist a $\la_0 \in \{0 \dts m\}$ such that $t_{\la_0} = 0$. Furthermore, $1 \leq 2^{\la_0} \leq i $ and  it is clear that from Theorem \ref{Lucas} $a_k(2^{\la_0},i) = 1$.
\end{proof}

\begin{proof}(\textbf{Lemma \ref{2-digits-1}}).
Suppose $x> 0$. Statements (1), (2), (3) can be simultaneously proven through induction on $t = \max\{m,n\}$. For $t = 1$, we get $m = n =1$, and  $b_1 > a_1$. So $x = 2^{b_1} - 2^{a_1} = 2^{b_1-1} +\cdots + 2^{a_1}$, thus, satisfies   statements (1), (2) and (3). Suppose now $t\geq 2$. For technical reasons there are some cases we need to distinguish, particularly those are $(m =1, \  n\geq 2)$, \ $(m \geq 2, \ n =1$)  and $(m\geq 2, \ n \geq 2)$. All of them can be treated similarly; thus, we will focus at the case $m,n\geq 2$. Given that $x>0$, we can assume $b_m \geq a_n$. \ 
 
If $b_m = a_n$, then $x = (2^{b_{m-1}}+\cdots + 2^{b_1})  - (2^{a_{n-1}} +\cdots + 2^{a_1})$, where $max\{m-1,n-1\} \leq t$. Consequently, statements (1),(2) and (3) follow through the induction process. \ 

If $b_m > a_n$, $a_1 = b_1$. Similarly as before, (1),(2) follow through the induction process. \ 

If  $b_m > a_n$, $a_1 < b_1$. Then 
\begin{align*}
x &= \sm_{ a_1 \leq  \la \leq b_m-1}2^{\la} + (2^{b_{m-1}} +\cdots + 2^{b_1}) - (2^{a_n} +\cdots + 2^{a_2}) \\
&= \sm_{\substack{a_1 < \la \leq b_m-1 \\ \la\neq a_2 \dts a_n}}2^{\la} + 2^{b_{m-1}}+\cdots + 2^{b_1} + 2^{a_1} = \sm_{\la \geq \min\{a_1 + 1, b_1\}}h_{\la}2^{\la} + 2^{a_1},     
\end{align*}
for some $h_{\la} \in \{0,1\}$. Therefore, (1),(2) and (3) follows. \ 

If  $b_m > a_n$, $a_1 > b_1$. As before, we can get, $x = \sm_{\la \geq \min\{a_1 , b_2\}}h_{\la}2^{\la} + 2^{b_1}$, and so $(1), (2)$ follows.
\end{proof}

\begin{proof}(\textbf{Lemma \ref{2-digits-2}}).
From Theorem \ref{Lucas}, it suffices to prove only the implication $(\Longrightarrow)$. Assume the contrary, namely, that $s$ does not contain the base 2 representation of $t$, to arrive at a contradiction. Then there exists $\la_1 \in \{ 0 \dts n\}$, such that $t_{\la_1} = 1, \ s_{\la_1} = 0$. In fact, we assume the minimum such $\la_1$, and thus $0\leq s_{\la} - t_{\la}$, for all $\la < \la_1$. Now, we proceed with the computation
\begin{equation*}
s-t = \sm_{\la  = \la_1 + 1}^ms_{\la}2^{\la} - \sm_{\la = \la_1 +1 }^nt_{\la}2^{\la} + 2^{\la_1} + \sm_{\la = 0}^{\la_1 - 1}(s_{\la} - t_{\la})2^{\la}.    
\end{equation*}
In case one of the following happens,  $\la_1 + 1 >m, \ \la_1 + 1 > n, \  \la_1-1 < 0$, we mean that the corresponding sum is zero. We set $y = \sm_{\la  = \la_1 + 1}^ms_{\la}2^{\la} - \sm_{\la = \la_1 +1 }^nt_{\la}2^{\la}$ and examine the following cases \ 

If $y < 0$. Then, by Lemma \ref{2-digits-1}, the element $-y> 0$, has a 2-digit $2^x$, where $x \geq \la_1 +1$. Therefore, $0 < s- t \leq -2^x + 2^{\la_1} + \sm_{\la = 0}^{\la_1 - 1}(s_{\la} - t_{\la})2^{\la} < 0$,  a contradiction.  

If $y > 0$. Then according to Lemma \ref{2-digits-1}, the element $y> 0$, can be expressed in a base 2 form $y = \sm_{\la \geq \la_1 +1}y_{\la}2^{\la}$. Combining this with the previous fact that $0 \leq s_{\la} - t_{\la} \leq 1$ for all $\la < \la_1$, we can deduce that $2^{\la_1}$ is 2-digit of $s-t$. However, $2^{\la_1}$ is not a 2-digit of $s$, and so  Theorem \ref{Lucas} implies $\epsilon(s-t,s) = 0$, a contradiction. \ 
The case $y= 0$, similarly  to the case $y > 0$, also leads to contradiction. 
\end{proof}

\begin{proof}(\textbf{Lemma \ref{Elementary1}}).
We will use induction on $k$. Consider the minimum $s_0 \in \{1 \dts k\}$, such that $2^{a_{s_0}}$ is a 2-digit of $x$, specifically, $2^{a_{s_0}}$ is not a 2-digit of $x + 2^{a_{s_0}}$. For $k = 1$, it's clear. Let us assume that $k \geq 2$ and the result holds for every $m \leq k-1$. We examine the following two cases \ 

\textbf{Case 1}. $x+2^{a_{s_0}}$ does not have 2-digit in the set $\{2^{a_i}: i\neq s_0\}$. \ 

 Set $s = 1$ and $b_1 = a_{s_0}$. Then, $x + 2^{b_1}$ does not have a 2-digit in the set  $\{2^{a_1} \dts 2^{a_k}\}$. It remains to prove uniqueness. Assume the existence of another such sequence, say $2^{c_1} <\cdots<  2^{c_r}$. We claim that $c_1 = a_{s_0}$. Indeed, if $c_1 < a_{s_0}$, observe by assumption on $s_0$, that $2^{c_1}$ is a 2-digit of $x+ 2^{c_r} +\cdots + 2^{c_1}$, where $c_1 \in \{a_1 \dts a_k\}$. If $c_1 > a_{s_0}$, observe that that $2^{a_{s_0}}$ is a 2-digit of $x+ 2^{c_r} +\cdots + 2^{c_1}$. Both of the o previous cases leads to contradiction and so $c_1 = a_{s_0}$. Consequently, $x + 2^{c_r} +\cdots + 2^{c_1} = (x+ 2^{a_{s_0}}) + 2^{c_r} +\cdots + 2^{c_2}$, where left hand side does not have a 2-digit in the set $\{2^{a_1} \dts 2^{a_k}\}$, and $x+ 2^{a_{s_0}}$ does not have a 2-digit in the set $\{2^{a_i}: i\neq s_0\}$. This forces $r = 1$ and $c_1 = a_{s_0}$. \

\textbf{Case 2}. $x+2^{a_{s_0}}$  have a 2-digit in the set $\{2^{a_i}: i\neq s_0\}$. \

 Set $y = x+2^{a_{s_0}}$ and consider the sequence $2^{a_1} <\cdots<  \widehat{2^{a_{s_0}}} <\cdots<  2^{a_k}$, (the $s_0$-term is omitted). From induction hypothesis there exists unique sequence $b_1 <\cdots<  b_s$ such that $b_i \in \{a_i : i \neq s_0\}$ and $y + 2^{b_t} +\cdots + 2^{b_1}$ does not have a 2-digit inside the set $\{2^{a_i}: i \neq s_0\}$. Now notice that $a_{s_0} < b_1$, for if we had $a_{s_0} > b_1$, then from the assumption on $s_0$, $2^{b_1}$ is a 2-digit of  $x + 2^{a_{s_0}} + 2^{b_s} +\cdots + 2^{b_1}$, which leads to contradiction. Inequalities $a_{s_0} < b_1 <\cdots<  b_s$ combined with the fact that $2^{a_{s_0}}$ is 2-digit of $x$, implies that $2^{a_{s_0}}$ is not a 2-digit of  $x  + 2^{b_s} +\cdots + 2^{b_1} + 2^{a_{s_0}}$. Moreover   $x  + 2^{b_s} +\cdots + 2^{b_1} + 2^{a_{s_0}}$ does not have a 2-digit in set $\{2^{a_1} \dts 2^{a_k}\}$. Uniqueness of sequence $( a_{s_0}, b_1 \dts b_s)$ remains to be proved.  Assume the existence of another such sequence lets say $2^{c_1} <\cdots<  2^{c_r}$. Similarly to previous case, we have $c_1 = a_{s_0}$. Hence, the element $y + 2^{c_r} +\cdots + 2^{c_2}$, does not have a 2-digit in the set $\{2^{a_i}: i \neq s_0\}$. Induction hypothesis, forces $(b_1 \dts b_s) = (c_2 \dts c_r)$.
\end{proof}

\begin{proof}(\textbf{Lemma \ref{Elementary2}}).
We are  splitting the proof in the following two cases $\la_1 = a_1$ and $\la_1 < a_1$. \ 

Starting with the more demanding case $\la_1 = a_1$, we employ induction o $n$. For $n = 1$, set $m = 1, \ b_1 = \la_2$. The  element $2^{b_1} - 2^{a_1} = 2^{\la_2-1} +\cdots + 2^{\la_1}$   does not have any 2-digit inside the set $\{2^{\la_2} \dts 2^{\la_{\nu}}\}$. As for the uniqueness, assume the existence of another such sequence lets say
$c_1 <\cdots<  c_k$. If $k > 1$ or ($k = 1$ and $c_1 > \la_2$), then it is easy see that the element $(2^{c_k} +\cdots + 2^{c_1})-2^{a_1}$ has 2-digits inside the set $\{2^{\la_2} \dts 2^{\la_{\nu}}\}$, which leads to contradiction.  \ 

 Now assume $n \geq 2$. Since  $2^{a_n} +\cdots + 2^{a_1} < 2^{\la_1} +\cdots + 2^{\la_{\nu}}$,  exactly one of the following situations is occurring

\begin{enumerate}

    \item 
There exists $s_0 \in \{2 \dts n\}$ such that $a_1 = \la_1 \dts a_{s_0-1} = \la_{s_0-1}$ and $a_{s_0} \neq \la_{s_0}$. Specifically, $a_{s_0} > \la_{s_0}$ (because $a_{s_0} \in \{\la_1 \dts \la_{\nu}\}$).
    
    \item 
$a_1 = \la_1 \dts a_n = \la_n$ and $n < \nu$.  
\end{enumerate}

We initiate by proving the existence of the sequence in each situation.\ 

If situation (1) occurs. By setting $(b_1 \dts b_m) = (\la_{s_0}, a_{s_0} \dts a_n)$, we compute that
\begin{align*}
(2^{b_m} +\cdots + 2^{b_1})-(2^{a_n} +\cdots + 2^{a_1}) &= 2^{\la_{s_0}}- 2^{\la_1} - (2^{\la_{s_0-1}} +\cdots + 2^{\la_2})  \\
&=\sm_{\substack{\la_1 \leq \la \leq \la_{s_0}-1: \\ \la \neq \la_{2} \dts \la_{s_0-1}}}2^{\la},    
\end{align*}
which does not contain 2-digit in the set $\{2^{\la_2} \dts 2^{\la_{\nu}}\}$. \ 

If situation (2) occurs. Set $m = 1$ and $b_1 = \la_{n+1}$. Then $2^{\la_{n+1}} - (2^{a_n} +\cdots + 2^{a_1}) = \sm_{\substack{\la_1 \leq \la \leq \la_{n+1}-1: \\ \la \neq \la_{2} \dts \la_{n}}}2^{\la}$, which does not contain 2-digit in the set $\{2^{\la_2} \dts 2^{\la_{\nu}}\}$. \\

 It remain to prove the uniqueness in its situation. Assume the existence of another such sequence lets say $c_1 <\cdots<  c_t$. We have to distinguish the following cases 

 \textbf{Case}. Situation (1) occurs for the sequence $a_1 <\cdots<  a_n$.  We initiate by observing that $t \geq 2$ and $c_t = a_n$.  \  
 
 Indeed, if we had $t = 1$, then by (1) the element $2^{c_1} -(2^{a_n} +\cdots + 2^{a_1}) = \sm_{\substack{\la_1 \leq \la \leq c_1-1: \\ \la \neq a_{2} \dts a_n}}2^{\la}$,  must have $2^{\la_{s_0}}$ as a 2-digit, which leads to contradiction. On the other hand, now assume $c_t \neq a_n$. Then $c_t > a_n$ and $(2^{c_t} +\cdots + 2^{c_1})-(2^{a_n} +\cdots + 2^{a_1})  = \sm_{\substack{\la_1 \leq \la \leq c_t: \\ \la \neq a_{2} \dts a_n, c_t}}2^{\la} + \sm_{\la = 1}^{t-1}2^{c_{\la}}$, the last expression, according to Lemma \ref{2-digit-3} have a 2-digit in the set $A= \{2^{a_2} \dts 2^{a_n},2^{c_t}, 2^{c_1} \dts 2^{c_{t-1}}\}$. If $\la_1 < c_1$, then $A \subseteq \{2^{\la_2} \dts 2^{\la_n}\}$. If $\la_1 = c_1 = a_1$, then by Lemma \ref{2-digits-1} the element $(2^{c_t} +\cdots + 2^{c_1})-(2^{a_n} +\cdots + 2^{a_1}) = (2^{c_t} +\cdots + 2^{c_2})-(2^{a_n} +\cdots + 2^{a_2})$ contains a 2-digit in the set $\{2^{\la_2} \dts 2^{\la_n}\}$.  Both cases comes in contrary  with the assumption on sequence $c_1 <\cdots<  c_t$. Consequently $c_t = a_n$. \ 

 As a result, $(2^{c_t} +\cdots + 2^{c_1})-(2^{a_n} +\cdots + 2^{a_1})  =  (2^{c_{t-1}} +\cdots + 2^{c_1})-(2^{a_{n-1}} +\cdots + 2^{a_1})$, where we can apply the induction step. However, to proceed with the induction, it is necessary to determine whether situation (1) or (2) occurs for the sequence $a_1 <\cdots<  a_{n-1}$.\ 

If $s_0 \leq n-1$. Then also situation (1) occurs for the sequence $a_1 <\cdots<  a_{n-1}$, and so the induction step implies $(c_1 \dts c_{t-1}) = (\la_{s_0}, a_{s_0} \dts a_{n-1})$. \ 

If $s_0 = n$.   Then situation (2) occurs  for the sequence $a_1 <\cdots<  a_{n-1}$, and so the induction step implies  $t-1 = 1, \ c_1 =  \la_{n-1+1} = \la_n$. Consequently, $t = 2, \ s_0 = n$ and $(c_1 \dts c_t) = (\la_{s_0}, a_{s_0} \dts a_n)$. \\

 \textbf{Case}. Situation (2) occurs for the sequence $a_1 <\cdots<  a_n$. Since $a_1 = \la_1 \dts a_n = \la_n$  and $2^{c_t} + \cdots +2^{c_1} > 2^{a_n} +\cdots + 2^{a_1}$, we get that $c_t > a_n = \la_n$. Furthermore, assuming that $t \geq 2$ we compute that  $(2^{c_t} +\cdots + 2^{c_1})-(2^{a_n} +\cdots + 2^{a_1})  = \sm_{\substack{\la_1 \leq \la \leq c_t: \\ \la \neq a_{2} \dts a_n, c_t}}2^{\la} + \sm_{\la = 1}^{t-1}2^{c_{\la}}$, which by Lemma \ref{2-digit-3} leads to contradiction. Hence, $t = 1$. Similarly, we can prove $c_1 = \la_{n+1}$, thereby confirming the uniqueness. \\ 

 It remains to examine case $\la_1 < a_1$, which appears to be easier; therefore, we provide a brief sketch. The requiring sequence is $(b_1 \dts b_m) = (\la_1,a_1 \dts a_n)$. As for the uniqueness, let's consider another such sequence, say $c_1 <\cdots<  c_t$. Then, by employing Lemma \ref{2-digits-1} we demonstrate the uniqueness.
\end{proof}

\begin{lemma}\label{2-digit-3}
Consider two sequences $b_0 < b_1 <\cdots<  b_s$, \ $b_0 \leq c_1 <\cdots<  c_r  < b_s$, then the element $\sm_{\substack{b_0 \leq \la \leq b_s: \\ \la \neq b_1 \dts b_s}}2^{\la} + \sm_{\la = 1}^r2^{c_{\la}}$, have a 2-digit in the set $\{2^{b_1} \dts 2^{b_s}, 2^{c_1} \dts 2^{c_r}\}$.
\end{lemma}

\begin{proof}
By induction on $r$.   
\end{proof}

\section*{Acknowledgments}

The author expresses his hearty thanks to Professor Mihalis Maliakas for suggesting the study of certain Hom-spaces. Also he acknowledges support of the Department of Mathematics, University of Athens.

\bibliographystyle{plain}
\bibliography{citation}

\begin{thebibliography}{10}

\bibitem{akin1985characteristic}
Kaan Akin and David~A Buchsbaum.
\newblock Characteristic-free representation theory of the general linear group.
\newblock {\em Advances in Mathematics}, 58(2):149--200, 1985.

\bibitem{akin1982schur}
Kaan Akin, David~A Buchsbaum, and Jerzy Weyman.
\newblock Schur functors and schur complexes.
\newblock {\em Advances in Mathematics}, 44(3):207--278, 1982.

\bibitem{CL}
Roger~W Carter and George Lusztig.
\newblock On the modular representations of the general linear and symmetric groups.
\newblock {\em Mathematische Zeitschrift}, 136:193--242, 1974.

\bibitem{CP}
RW~Carter and MTJ Payne.
\newblock On homomorphisms between weyl modules and specht modules.
\newblock In {\em Mathematical Proceedings of the Cambridge Philosophical Society}, volume~87, pages 419--425. Cambridge University Press, 1980.

\bibitem{CoxParkerGL3}
Anton Cox and Alison Parker.
\newblock Homomorphisms and higher extensions for schur algebras and symmetric groups.
\newblock {\em Journal of Algebra and Its Applications}, 4(06):645--670, 2005.

\bibitem{dodge2011large}
Craig~J Dodge.
\newblock Large dimension homomorphism spaces between specht modules for symmetric groups.
\newblock {\em Journal of Pure and Applied Algebra}, 215(12):2949--2956, 2011.

\bibitem{ellers2010carter}
Harald Ellers and John Murray.
\newblock Carter--payne homomorphisms and branching rules for endomorphism rings of specht modules.
\newblock 2010.

\bibitem{evangelou2023stability}
Charalambos Evangelou, Mihalis Maliakas, and Dimitra-Dionysia Stergiopoulou.
\newblock On stability and nonvanishing of homomorphism spaces between weyl modules.
\newblock {\em Algebraic Combinatorics (to appear)}, 2023.

\bibitem{green2006polynomial}
James~A Green.
\newblock {\em Polynomial Representations of GL\_n: with an Appendix on Schensted Correspondence and Littelmann Paths}, volume 830.
\newblock Springer, 2006.

\bibitem{james2006representation}
Gordon~Douglas James.
\newblock {\em The representation theory of the symmetric groups}, volume 682.
\newblock Springer, 2006.

\bibitem{jantzen2003representations}
Jens~Carsten Jantzen.
\newblock {\em Representations of algebraic groups}, volume 107.
\newblock American Mathematical Soc., 2003.

\bibitem{kulkarni2006ext}
Upendra Kulkarni.
\newblock On the ext groups between weyl modules for gln.
\newblock {\em Journal of Algebra}, 304(1):510--542, 2006.

\bibitem{lucas1878theorie}
Edouard Lucas.
\newblock Th{\'e}orie des fonctions num{\'e}riques simplement p{\'e}riodiques.[continued].
\newblock {\em American Journal of Mathematics}, 1(3):197--240, 1878.

\bibitem{lyle2013large}
Sin{\'e}ad Lyle.
\newblock Large-dimensional homomorphism spaces between weyl modules and specht modules.
\newblock {\em Journal of Pure and Applied Algebra}, 217(1):87--96, 2013.

\bibitem{lyle2013homomorphisms}
Sin{\'e}ad Lyle.
\newblock On homomorphisms indexed by semistandard tableaux.
\newblock {\em Algebras and Representation Theory}, 16(5):1409--1447, 2013.

\bibitem{maliakas2021homomorphisms}
Mihalis Maliakas and Dimitra-Dionysia Stergiopoulou.
\newblock On homomorphisms involving a hook weyl module.
\newblock {\em Journal of Algebra}, 585:1--24, 2021.

\bibitem{maliakas2023homomorphisms}
Mihalis Maliakas and Dimitra-Dionysia Stergiopoulou.
\newblock On homomorphisms into weyl modules corresponding to partitions with two parts.
\newblock {\em Glasgow Mathematical Journal}, 65(2):272--283, 2023.

\end{thebibliography}

\end{document}